\documentclass[a4paper, DIV=12, 11pt]{amsart}
\usepackage[latin1]{inputenc}
\usepackage{amssymb, amsthm, amsmath, thmtools, mathtools}
\usepackage{amsfonts}
\usepackage[abbrev]{amsrefs} 

\usepackage{graphicx}
\usepackage{thmtools}
\usepackage{hyperref}
\usepackage[bottom, marginal]{footmisc}
\usepackage{todonotes}

\usepackage{tikz}
\usetikzlibrary{matrix}

\declaretheoremstyle[headfont=\normalfont]{normalhead}
\newtheorem{lemma}{Lemma}[section]
\newtheorem{theorem}[lemma]{Theorem}
\newtheorem{proposition}[lemma]{Proposition}
\newtheorem{corollary}[lemma]{Corollary}
\newtheorem{definition}[lemma]{Definition}
\newtheorem{remark}[lemma]{Remark}

\newtheorem{question}{Question}

\newcounter{mt}

\newtheorem{maintheorem}[mt]{Theorem}
\newtheorem{maincorollary}[mt]{Corollary}

\newtheorem*{acknowledgement}{Acknowledgement}

\newcommand{\R}{\mathbb{R}}
\newcommand{\C}{\mathbb{C}}
\newcommand{\U}{\mathrm{U}}

\DeclareMathOperator{\sign}{sign}
\DeclareMathOperator{\Val}{Val}

\DeclareMathOperator{\Conv}{Conv}
\DeclareMathOperator{\vol}{vol}
\DeclareMathOperator{\vsupp}{v-supp}
\DeclareMathOperator{\supp}{supp}

\DeclareMathOperator{\epi}{epi}

\DeclareMathOperator{\diam}{diam}

\DeclareMathOperator{\GL}{GL}

\DeclareMathOperator{\GW}{\mathrm{GW}}

\DeclareMathOperator{\Gr}{\mathrm{Gr}}

\DeclareMathOperator{\Sym}{\mathrm{Sym}}

\DeclareMathOperator{\nc}{\mathrm{nc}}

\DeclareMathOperator{\SO}{\mathrm{SO}}

\DeclareMathOperator{\MAVal}{\mathrm{MAVal}}

\DeclareMathOperator{\LV}{\mathrm{LV}}

\newcommand{\M}{\mathcal{M}}

\newcommand{\K}{\mathcal{K}}
\newcommand{\vsupploc}{\vsupp_{\mathrm{loc}}}
\newcommand{\supploc}{\mathrm{supp}_{\mathrm{loc}}}

\DeclareMathOperator{\Area}{\mathrm{Area}}

\numberwithin{equation}{section}

\author{Jonas Knoerr}
\title[Translation invariant area measures on convex bodies]{Translation invariant area measures on convex bodies}
\date{}

\newcommand{\Addresses}{{
		\bigskip
		\footnotesize
		
		Jonas Knoerr, \textsc{Institute of Discrete Mathematics and Geometry, TU Wien, Wiedner Hauptstrasse 8-10, 1040 Wien, Austria}\par\nopagebreak
		\textit{E-mail address}: \texttt{jonas.knoerr@tuwien.ac.at}
		
		\medskip
	}}
	
\makeatletter
\def\blfootnote{\xdef\@thefnmark{}\@footnotetext}
\makeatother

\makeindex
\begin{document}
\maketitle
\begin{abstract}
	We introduce the space of continuous and translation invariant area measures, which are measure-valued functionals on the space of convex bodies satisfying a certain locality condition. Our main result shows that the space of $\GL(n,\R)$-smooth area measures coincides with the space of area measures obtained by integration with respect to the normal cycle. We show how this result yields Hadwiger-type classification results for continuous area measures that are equivariant with respect to compact groups acting transitively on the unit sphere. In addition, we establish a general density criterion for invariant submodules and show that mixed area measures generate dense submodules with respect to suitable topologies on the space of continuous area measures. As a byproduct, we discuss how McMullen's Conjecture can be obtained directly from the representation of $\GL(n,\R)$-smooth translation invariant valuations on convex bodies in terms of integration with respect to the normal cycle. 
\end{abstract}


\section{Introduction}

Let $\mathcal{K}(\R^n)$ denote the space of convex bodies in $\R^n$, i.e. the set of all nonempty, compact, and convex subsets of $\R^n$ equipped with the Hausdorff metric. One way to encode the local geometry of the boundary of a convex body are the surface area measure $S_{n-1}(K)$ and the related mixed area measures, which feature prominently in various problems in integral and convex geometry. The surface area measure may be characterized as the unique nonnegative measure on the Euclidean unit sphere $S^{n-1}$ satisfying
\begin{align}
	\label{eq:definitionSurfacteArea}
	\frac{d}{dt}\Big|_0\vol(K+tL)=\int_{S^{n-1}}h_L dS_{n-1}(K)\quad\text{for}~L\in\K(\R^n),
\end{align}
where $h_L(y)=\sup_{x\in L}\langle x,y\rangle$, $y\in\R^n$, denotes the support function of $L\in\K(\R^n)$. The mixed area measure $S(K_1,\dots,K_{n-1})$ of $K_1,\dots,K_{n-1}\in \K(\R^n)$ is given by
\begin{align*}
	S(K_1,\dots,K_{n-1})=\frac{1}{(n-1)!}\frac{\partial^{n-1}}{\partial\lambda_1\dots\partial\lambda_{n-1}}\Big|_0 S_{n-1}\left(\sum_{j=1}^{n-1}\lambda_j K_j\right).
\end{align*}
In particular, $S_i(K):=S(K[i],B_1(0)[n-1-i])$ is called the $i$th intermediate area measure, where the brackets $[k]$ indicate that the relevant body is taken with multiplicity $k$ in the mixed area measure and $B_1(0)$ is the Euclidean unit ball. 

We may consider these functionals as maps $\K(\R^n)\rightarrow\M(S^{n-1})$ into the space of finite complex Borel measures on $S^{n-1}$. As measure-valued functionals, the intermediate area measures admit an axiomatic characterization in terms of their invariance and locality properties, as shown by Schneider \cite{SchneiderKinematischeBeruhrmaefur1975}. We postpone the discussion of his result to \autoref{section:LocalityCondition}, since his classification relies on slightly stronger properties than required for our purposes. The aim of this article is a characterization of certain measure-valued functionals on $\K(\R^n)$ that share more general properties of the surface area measure. More precisely, we consider the following space.
\begin{definition}\label{def:Area}
	We denote by $\Area(\R^n)$ the space of all maps $\Psi:\K(\R^n)\rightarrow \M(S^{n-1})$ with the following properties:
	\begin{enumerate}
		\item $\Psi$ is continuous with respect to the weak* topology.
		\item $\Psi$ is locally determined: If $K,L\in\K(\R^n)$ satisfy $h_K|_U=h_L|_U$ for an open set $U\subset S^{n-1}$, then the associated measures satisfy $\Psi(K)|_U=\Psi(L)|_U$.
		\item $\Psi$ is translation invariant: $\Psi(K+x)=\Psi(K)$ for all $K\in\K(\R^n)$ and $x\in\R^n$.
	\end{enumerate}
\end{definition}
Note that any finite Borel measure on $S^{n-1}$ is a Radon measure, so we may identify $\M(S^{n-1})\cong C(S^{n-1})'$ with the topological dual of the space of continuous functions on $S^{n-1}$ with respect to the topology of uniform convergence. The second property can be interpreted as a certain locality condition: It implies that the measures coincide on open sets such that the parts of the boundaries of the convex bodies with outer normals belonging to these open sets are the same, compare \autoref{section:preliminaries_convexbodies}.

We will call elements in $\Area(\R^n)$ continuous translation invariant \emph{area measures}, a terminology that will be justified by \autoref{maintheorem:finiteMixedArea} below. A priori, our notion of area measures differs slightly from other articles (compare \cite{AbardiaEvequozEtAlFlagareameasures2019,AbardiaEvequozBernigAdditivekinematicformulas2021,BernigDualareameasures2019,WannererIntegralgeometryunitary2014,Wannerermoduleunitarilyinvariant2014}), which usually emphasize that $\Psi:\K(\R^n)\rightarrow \M(S^{n-1})$ should be a measure-valued valuation, i.e. that
\begin{align*}
	\Psi(K\cup L)+\Psi(K\cap L)=\Psi(K)+\Psi(L)
\end{align*}
for all $K,L\in\K(\R^n)$ such that $K\cup L$ is convex. However, we will show in \autoref{theorem:valuationProperty} that a continuous and locally determined functional automatically satisfies the valuation property.  In particular, this relation to the theory of valuations on convex bodies directly implies several structural results. Let us denote by $\Area_k(\R^n)$ the subspace of all
$\Psi\in \Area(\R^n)$ such that $\Psi(tK)=t^k\Psi(K)$ for $t\ge 0$ and $K\in\mathcal{K}(\R^n)$.
\begin{maintheorem}\label{maintheorem:DirectClassificationResults}
	We have a direct sum decomposition
	\begin{align*}
		\Area(\R^n)=\bigoplus_{k=0}^{n-1}\Area_k(\R^n).
	\end{align*}
	Moreover, the following holds:
	\begin{enumerate}
		\item the map $\Psi\mapsto \Psi(0)$ establishes an isomorphism $\Area_0(\R^n)\cong\M(S^{n-1})$.
		\item $\Psi\in\Area_{n-1}(\R^n)$ if an only if there exists a function $\psi\in C(S^{n-1})$ such that $\Psi(K;\phi)=\int_{S^{n-1}}\phi\cdot\psi dS(K)$. This function is uniquely determined.
	\end{enumerate}
\end{maintheorem}
The first part is based on McMullen's homogeneous decomposition for translation invariant valuations \cite{McMullenValuationsEulertype1977}, while (2) can be seen as an analog of McMullen's characterization of $(n-1)$-homogeneous, continuous, and translation invariant valuations on $\K(\R^n)$ \cite{McMullenContinuoustranslationinvariant1980}.\\

The main result of this article is a description of the components of intermediate degree of homogeneity in terms of a classification of a dense subspace. We are going to consider several natural topologies on $\Area(\R ^n)$, however, for the introduction we focus on the following: For $\phi\in C(S^{n-1})$, consider the semi-norm on $\Area(\R^n)$ defined by
\begin{align*}
	\Psi\mapsto \sup_{K\subset B_1(0)}|\Psi(K;\phi)|,
\end{align*}
where $\phi\mapsto \Psi(K;\phi)$ denotes the integration functional on $C(S^{n-1})$ induced by $\Psi(K)$. We call the topology on $\Area(\R^n)$ determined by these semi-norms for all $\phi\in C(S^{n-1})$ the uniform weak* topology.\\

Let us define an action of the general linear group $\GL(n,\R)$ on $C(S^{n-1})$ by 
\begin{align*}
	[g\cdot \phi](y)=|g^{T}v| \phi\left(\frac{g^{T}v}{|g^{T}v|}\right),\quad v\in S^{n-1},
\end{align*}
for $\phi\in C(S^{n-1})$, $g\in\GL(n,\R)$. In other words, we identify elements in $C(S^{n-1})$ with $1$-homogeneous functions on the dual space of $\R^n$. We obtain a representation $\pi$ of $\GL(n,\R)$ on $\Area(\R^n)$ given by
\begin{align*}
	[\pi(g)\Psi](K;\phi)=\Psi(g^{-1}K;g^{-1}\cdot \phi).
\end{align*}
For example,  Eq.~\eqref{eq:definitionSurfacteArea} shows that 
\begin{align}
	\label{eq:GroupActionSurfaceArea}
	\pi(g)S_{n-1}=|\det(g)|^{-1}S_{n-1}.
\end{align}
Let us call $\Psi\in\Area(\R^n)$ $\GL(n,\R)$-smooth if the map
\begin{align*}
	\GL(n,\R)&\rightarrow\Area(\R^n)\\
	g&\mapsto \pi(g)\Psi
\end{align*}
is smooth with respect to the uniform weak* topology. It is well known that the $\GL(n,\R)$-smooth vectors of a representation form a dense subspace under suitable completeness and continuity conditions on the representation. This also applies to the representation of $\GL(n,\R)$ on $\Area(\R^n)$, compare the discussion and references in \autoref{section:ActionGroup}.\\

In order to describe this space, we need a different notion of smoothness for elements in $\Area(\R^n)$ considered in \cite{BernigDualareameasures2019,WannererIntegralgeometryunitary2014,Wannerermoduleunitarilyinvariant2014}. Let $\nc(K)$ denote the normal cycle of $K\in\K(\R^n)$ (see \autoref{section:ConstructionSmoothArea} for details). This is an $(n-1)$-dimensional Lipschitz submanifold of the sphere bundle $\R^n\times S^{n-1}$ and inherits a natural orientation from a chosen orientation of $\R^n$ (which we take to be the standard orientation for simplicity). In particular, $\nc(K)$ may be considered as an $(n-1)$-dimensional integral current on $\R^n\times S^{n-1}$. Let $\Omega^{l}(\R^n\times S^{n-1})^{tr}$ denote the space of smooth differential $l$-forms on $\R^n\times S^{n-1}$ that are invariant with respect to translations. If $\pi_2:\R^n\times S^{n-1}\rightarrow S^{n-1}$ denotes the projection onto the second factor, we may associate to any $K\in\K(\R^n)$, $\omega\in \Omega^{n-1}(\R^n\times S^{n-1})^{tr}$ the measure $\Phi_\omega(K)\in \M(S^{n-1})$ defined by
\begin{align*}
	\Phi_\omega(K;\phi):=\int_{\nc(K)}\pi_2^*\phi \wedge\omega\quad\text{for}~\phi\in C(S^{n-1}).
\end{align*}
Functionals of this form were called \emph{smooth area measures} in \cite{BernigDualareameasures2019,WannererIntegralgeometryunitary2014,Wannerermoduleunitarilyinvariant2014}, which is motivated by a corresponding description of the space of $\GL(n,\R)$-smooth translation invariant valuations on $\K(\R^n)$ \cite{AleskerTheoryvaluationsmanifolds.2006,HofstaetterKnoerrLocalizationvaluationsAleskers2025}, but is at its core a purely synthetic notion. Our next result shows that this coincides with our representation theoretic definition.

\begin{maintheorem}
	\label{maintheorem:IntegralRepresentation}
	Let $\Psi\in\Area(\R^n)$. The following are equivalent:
	\begin{enumerate}
		\item $\Psi$ is a $\GL(n,\R)$-smooth area measure.
		\item There exists a differential form $\omega\in \Omega^{n-1}(\R^n\times S^{n-1})^{tr}$ such that $\Psi=\Phi_\omega$.
	\end{enumerate}
\end{maintheorem}
Since $\GL(n,\R)$-smooth area measures are dense in $\Area(\R^n)$ with respect to the uniform weak* topology, this result can be used to extend results from area measures given by integration with respect to the normal cycle to arbitrary continuous area measures. As an example, we consider certain injectivity results for moment maps original established in \cite{BernigDualareameasures2019} (compare \autoref{section:momentMaps}). More importantly, it can be used to obtain Hadwiger-type classification results for certain compact subgroups.
	\begin{maincorollary}\label{maincorollary:InvariantAreaMeasures}
	Let $G\subset \SO(n)$ be a compact subgroup that operates transitively on $S^{n-1}$. If $\Psi:\K(\R^n)\rightarrow\M(S^{n-1})$ is 
	\begin{enumerate}
		\item translation invariant,
		\item $G$-equivariant: $\Psi(gK;gB)=\Psi(K;B)$ for $K\in\K(\R^n)$, $B\subset S^{n-1}$ Borel set, $g\in G$,
		\item continuous with respect to the weak* topology,
		\item locally determined,
	\end{enumerate}
	then exists a $G$-invariant differential form $\omega\in\Omega^{n-1}(\R^n\times S^{n-1})^{tr}$ such that $\Psi(K;B)=\int_{\nc(K)\cap \pi_2^{-1}(B)}\omega$. In particular, the space $\Area(\R^n)^G$ of $G$-equivariant continuous area measures is finite dimensional and every element is $\GL(n,\R)$-smooth.
\end{maincorollary}
The compact subgroups that operate transitively and effectively on some unit sphere belong to one of the infinite series
\begin{align*}
	\SO(n), \U(n), \mathrm{SU}(n), \mathrm{Sp}(n),\mathrm{Sp}(n)\U(1),\mathrm{Sp}(n)\mathrm{Sp}(1)
\end{align*}
or are one of the exceptional cases
\begin{align*}
	\mathrm{G}_2,\mathrm{Spin}(7),\mathrm{Spin}(9),
\end{align*}
compare \cite{BorelSomeremarksLie1949a,MontgomerySamelsonTransformationgroupsspheres1943}. In combination with a description of the kernel of the normal cycle due to Bernig and Bröcker \cite{BernigBroeckerValuationsmanifoldsRumin2007}, \autoref{maincorollary:InvariantAreaMeasures} provides a direct way to obtain complete characterization results of all continuous $G$-equivariant area measures. For $G=\SO(n)$, it is easy to check that this recovers Schneider's result from \cite{SchneiderKinematischeBeruhrmaefur1975} (compare \autoref{theorem:SchneiderArea} below) in slightly improved form, while the case $G=\U(n)$ (acting on $\C^n$) follows from results by Wannerer \cite{WannererIntegralgeometryunitary2014,Wannerermoduleunitarilyinvariant2014} and Bernig and Fu \cite{BernigFuHermitianintegralgeometry2011}. Similarly, for the groups $\mathrm{SU}(n)$, $\mathrm{Sp}(2)\mathrm{Sp}(1)$, and $\mathrm{Spin}(9)$, the relevant differential forms can be found in \cite{BernigHadwigertypetheorem2009,BernigSolanesKinematicformulasquaternionic2017,KotrbatyWannererIntegralgeometryoctonionic2025}.\\

$\GL(n,\R)$-smooth area measures also admit the following description in terms of mixed area measures.
\begin{maintheorem}
	\label{maintheorem:finiteMixedArea}
	There exist ellipsoids $\mathcal{E}_j$, $1\le j\le N:=\binom{n+1}{2}+1$, such that for every $\GL(n,\R)$-smooth area measure $\Psi\in \Area_k(\R^n)$ there exist smooth functions $\phi_\alpha\in C^\infty(S^n)$ for $\alpha\in\mathbb{N}^N$ with $|\alpha|=n-k-1$ such that
	\begin{align*}
		\Psi(K;B)=\sum_{\alpha\in\mathbb{N}^N,|\alpha|=n-k-1}\int_{B}\phi_\alpha dS(K[k],\mathcal{E}_1[\alpha_1],\dots,\mathcal{E}_{n-k-1}[\alpha_{n-k-1}]).
	\end{align*}
\end{maintheorem}
In particular, any continuous area measure can be approximated by mixed surface area measures modified by a smooth density. This may be interpreted as a version of McMullen's Conjecture for the space $\Area(\R^n)$, which concerns a corresponding approximation result for translation invariant valuations on convex bodies in terms of mixed volumes, compare \cite{McMullenContinuoustranslationinvariant1980}. This conjecture was proved by Alesker in \cite{AleskerDescriptiontranslationinvariant2001} as a consequence of the Irreducibility Theorem, which provides a representation theoretic description of the space of continuous and translation invariant valuations on convex bodies. We show a related result for the spaces $\Area_k(\R^n)$ exploiting the additional $C(S^{n-1})$-module structure on $\Area(\R^n)$ given as follows: For $\psi\in C(S^{n-1})$ and  $\Psi\in\Area(\R^n)$, define $\psi\bullet\Psi\in\Area(\R^n)$ by
\begin{align*}
	[\psi\bullet \Psi](K;\phi):=\Psi(K;\psi\cdot\phi)=\int_{S^{n-1}}\psi\cdot\phi d\Psi(K).
\end{align*}
Thus, \autoref{maintheorem:DirectClassificationResults} shows that $\Area_{n-1}(\R^n)$ is a free $C(S^{n-1})$ module of rank $1$, while \autoref{maintheorem:finiteMixedArea} shows that the space of smooth area measures is generated as a module over $C^\infty(S^{n-1})$ by a finite family of mixed area measures. In the general case, we have the following density result.
\begin{maintheorem}
	\label{maintheorem:IrreducibleModule}
	Let $W\subset \Area_k(\R^n)$ be a nontrivial $\GL(n,\R)$-invariant $C(S^{n-1})$-submodule. Then $W$ is dense in $\Area_k(\R^n)$ with respect to the uniform weak* topology.
\end{maintheorem}

\subsection*{Plan of the article}
	The main contribution of this article is the description of $\GL(n,\R)$-smooth area measures in \autoref{maintheorem:IntegralRepresentation}. The proof of this result is based on a corresponding regularity result for a certain space of measure-valued functionals on convex functions recently established in \cite{KnoerrIntegralrepresentationpolynomial2026}. We will treat this as a "local version" of our main result and consequently apply a localization procedure to relate the two spaces of functionals, similar to the approach taken in \cite{HofstaetterKnoerrLocalizationvaluationsAleskers2025}. This requires a notion of support, which we introduce in \autoref{section:HomDecompArea} based on certain distributions introduced by Goodey and Weil \cite{GoodeyWeilDistributionsvaluations1984}, as well as some compatibility properties of the $C(S^{n-1})$-module structure with $\GL(n,\R)$-smooth vectors.\\
	This approach has one minor difficulty: It turns out that some of the constructions are only continuous with respect to a stronger topology on $\Area(\R^n)$ (called the \emph{uniform strong topology} below). However, with respect to this topology, the action of $\GL(n,\R)$ on $\Area(\R^n)$ does not have the necessary continuity properties required for certain arguments by approximation. Fortunately, this problem can be avoided by switching between different topologies on $\Area(\R^n)$, however, it does require some additional care in the arguments.\\
	
	The article is structured as follows. In \autoref{section:preliminaries}, we review the required background on convex bodies and valuations on convex bodies, in particular, the construction of the Goodey--Weil distributions. We also discuss the construction of area measures with the normal cycle, which requires some background from contact geometry, and show how this construction can be interpreted in a $\GL(n,\R)$-equivariant way. In addition, we prove an analog of \autoref{maintheorem:IrreducibleModule} for a certain space of differential forms.\\
	In \autoref{section:RelationValuationProperty}, we establish that a continuous locally determined functional is automatically a valuation, which we then use in \autoref{section:HomDecompArea} to obtain a proof of \autoref{maintheorem:DirectClassificationResults}. This gives rise to Goodey--Weil distributions for elements in $\Area_k(\R^n)$, and we modify these distributions to obtain a more suitable notion of support. The remaining subsections of \autoref{section:AreaGeneral} are used to discuss the different topologies required for the subsequent constructions and to verify various compatibility properties of these topologies with the action of $\GL(n,\R)$ and the $C(S^{n-1})$-module structure. Most of these results are a direct consequence of general properties of measure-valued functionals established in \cite{KnoerrPolynomiallocalfunctionals2025} and we refer to this article for a more in depth discussion.\\
	\autoref{section:CharcSmoothArea} contains the proof of \autoref{maintheorem:IntegralRepresentation}. We first show that a certain subspace of $\Area(\R^n)$ defined in terms of a support restriction can be identified with a space of local functionals on convex functions. This identification turns out to be equivariant with respect to the action of a certain subgroup of $\GL(n,\R)$, which we use to apply a regularity result from \cite{KnoerrIntegralrepresentationpolynomial2026}. As a consequence, we obtain \autoref{maincorollary:InvariantAreaMeasures}.\\
	\autoref{section:Applications} contains the remaining applications. First, we show that \autoref{maintheorem:IrreducibleModule} can be reduced to a stronger result for $\GL(n,\R)$-smooth area measures, which we obtain from the above mentioned analog result for differential forms. This result also allows for a rather direct proof of the representation of $\GL(n,\R)$-smooth area measures in terms of mixed area measures, compare \autoref{maintheorem:finiteMixedArea}. Interestingly, the same argument applies to $\GL(n,\R)$-smooth translation invariant valuations on convex bodies, which can be used to obtain McMullen's Conjecture without applying the Irreducibility Theorem \cite{AleskerDescriptiontranslationinvariant2001}. Since this is of independent interest, we include a short discussion in \autoref{section:FiniteComb}.\\
	In \autoref{section:LocalityCondition}, we discuss the relation of  the slightly stronger conditions used in Schneider's characterization of the intermediate area measures to out definition and establish some additional regularity properties for a subspace of $\Area(\R^n)$. Finally, we extend the injectivity results for moment maps obtained in \cite{BernigDualareameasures2019} to arbitrary continuous area measures.\\
	We close the article with some additional remarks and open questions in \autoref{section:remarks} .

	\begin{acknowledgement}
	This research was funded in whole or in part by the Austrian Science Fund (FWF), \href{https://www.doi.org/10.55776/PAT4205224}{10.55776/PAT4205224}.
\end{acknowledgement}

\section{Preliminaries}\label{section:preliminaries}
	Since every finite complex Borel measure on $S^{n-1}$ is a Radon measure, we will identify element in $\M(S^{n-1})$ with $C(S^{n-1})'$ without further comments. In particular, we will implicitly use that two such measures coincide on bounded Borel functions as soon as they coincide on continuous functions.\\
	Unless stated otherwise, all vector spaces are assumed to be complex vector spaces. If $V$ is a real vector space, we let $V_\C=V\otimes_\R\C$ denote its complexification.\\
	Throughout the article, we use $\omega_k$ to denote the $k$-dimensional volume of the Euclidean unit ball in $\R^k$.	

	\subsection{Convex bodies and functions}\label{section:preliminaries_convexbodies}
		As a general reference for convex bodies and functions, we refer to the monographs by Schneider \cite{SchneiderConvexbodiesBrunn2014} and Rockafellar \cite{RockafellarConvexanalysis1997}.
		Let us denote by $\Conv(\R^n,\R)$ the space of all convex functions $f:\R^n\rightarrow\R$. This space carries a natural (metrizable) topology induced by locally uniform convergence, which coincides with the topology induced by pointwise or epi-convergence convergence, see \cite{RockafellarConvexanalysis1997}*{Theorem~7.17}. For our purposes, the most important examples are support functions. 
		We collect the following properties of these functions.
		\begin{lemma}\label{lemma:propertiesSupportFunction}
			\begin{enumerate}
				\item $K_j\rightarrow K$ if and only if $h_{K_j}\rightarrow h_K$ locally uniformly on $\R^n$.
				\item $h_K\vee h_L=h_{\mathrm{conv}(K\cup L)}$, and if $K\cup L$ is convex, then $h_K\wedge h_L=h_{K\cap L}$, where $\vee$ and $\wedge$ denote the pointwise maximum and minimum.
				\item For $K\in\K(\R^n)$, $x\in\R^n$,
					\begin{align*}
						h_{K+x}(y)=h_K(y)+\langle y,x\rangle\quad\text{for}~y\in\R^n.
					\end{align*}
				\item For $K\in\K(\R^n)$, $g\in\GL(n,\R)$,
					\begin{align*}
						h_{gK}(y)=h_K(g^Ty)\quad \text{for}~y\in\R^n.
					\end{align*}
			\end{enumerate}
		\end{lemma}
		For $f\in\Conv(\R^n,\R)$ , the set 
		\begin{align}
			\label{eq:defSubgradient}
			\partial f(x_0)=\{v\in\R^n:f(x)\ge \langle v,x-x_0\rangle +f(x_0)\quad \text{for all}~x\in \R^n\}
		\end{align}
		is called the subdifferential of $f$ in $x_0$.  Note that this is a local notion for $f\in\Conv(\R^n,\R)$: If $v\in\R^n$ is not a subgradient of $f$ in $x_0$, then there exists $x\in\R^n$ such that $f(x)<f(x_0)+\langle v,x-x_0\rangle$. By convexity, for $t\in(0,1)$,
		\begin{align*}
			f(tx+(1-t)x_0)\le tf(x)+(1-t)f(x_0)<f(x_0)+t\langle v,x-x_0\rangle,
		\end{align*}
		so any of the points $tx+(1-t)x_0$, $t\in(0,1)$, may be used to check that $v$ is not a subgradient of $f$ in $x_0$. Thus, it is sufficient to check the inequality in Eq.~\eqref{eq:defSubgradient} on a neighborhood of $x_0$. In particular, if $f,h\in\Conv(\R^n,\R)$ coincide on a neighborhood of $x_0$, then $\partial f(x_0)=\partial h(x_0)$.\\

		For $K\in\K(\R^n)$ and $v\in\R^n$, we consider the support set
		\begin{align*}
			F(K,v):=\{x\in K: \langle x,v\rangle=h_K(v)\}.
		\end{align*}
		Note that this is a convex subset of the boundary of $K$ unless $v=0$. The support sets are related to the subdifferential of the support function in the following way.
		\begin{lemma}[\cite{SchneiderConvexbodiesBrunn2014}*{Theorem~1.7.4}]
			\label{lemma:SubdifferentialSupportFunction}
			For $K\in\K(\R^n)$ and $v\in\R^n\setminus\{0\}$, $\partial h_K(v)=F(K,v)$.
		\end{lemma}
		This implies in particular that for $K,L\in\K(\R^n)$ satisfying $h_K|_U=h_L|_U$ for an open subset $U\subset S^{n-1}$, we have $F(K,u)=F(L,u)$ for all $u\in U$. Here we used that the support functions are $1$-homogeneous, so the support functions actually coincide on an open neighborhood of $U$ in $\R^n$. 
		\begin{corollary}\label{corollary:RelationSubdiffNormalCycle}
			For $(x,v)\in\R^n\times \R^n$, the following are equivalent:
			\begin{enumerate}
				\item $x\in \partial h_K(v)$.
				\item $\langle v,x-y\rangle\ge 0$ for all $y\in K$.
			\end{enumerate}
		\end{corollary}
		\begin{proof}
			If $x\in \partial h_K(v)$, then \autoref{lemma:SubdifferentialSupportFunction} implies $h_K(v)=\langle v,x\rangle$, and so 
			\begin{align*}
				\langle v,x-y\rangle=h_K(v)-\langle v,y\rangle\ge 0\quad \text{for}~y\in K
			\end{align*}
			by the definition of the support function. Similarly, $\langle v,x-y\rangle\ge 0$ for all $y\in K$ implies that $\langle v,x\rangle \ge\sup_{y\in K} \langle v,y\rangle= h_K(v)$, so $x\in F(K,v)$ and we may apply \autoref{lemma:SubdifferentialSupportFunction} to obtain the result.
		\end{proof}
	\subsection{Valuations on convex bodies}\label{section:Valuations}
		For a Hausdorff topological vector space $F$, let us denote by $\Val(\R^n,F)$ the space of all continuous valuations $\mu:\K(\R^n)\rightarrow F$ that are translation invariant. We also consider the subspace  $\Val_k(\R^n,F)$ of $k$-homogeneous valuations, i.e the subset of all valuations $\mu\in\Val(\R^n,F)$ such that $\mu(tK)=t^k\mu(K)$ for $t\ge0$, $K\in\K(\R^n)$. We have the following homogeneous decomposition.
		\begin{theorem}[\cite{McMullenValuationsEulertype1977}]
			\label{theorem:McMullenHomDecomp}
			If $F$ is a Hausdorff topological vector space, then
			\begin{align*}
				\Val(\R^n,F)=\bigoplus_{k=0}^n \Val_k(\R^n,F).
			\end{align*}
		\end{theorem}
		This has the following consequence: If $F$ is a locally convex vector space, then we equip $\Val(\R^n,F)$ with the topology of uniform convergence on compact subset in $\K(\R^n)$. By the Blaschke Selection Theorem, a subset of $\K(\R^n)$ is compact if and only if it is closed and bounded in the Hausdorff metric (with respect to any Euclidean structure on $\R^n$). \autoref{theorem:McMullenHomDecomp} implies that the topology of $\Val(\R^n,F)$ is generated by the semi norms
		\begin{align*}
			\|\mu\|_F:=\sup_{K\subset L} |\mu(K)|_F
		\end{align*}
		for all continuous semi norms $|\cdot|_F$ on $F$ and a fixed full dimensional convex body $L\in\K(\R^n)$. We will use the unit ball with respect to the chosen Euclidean structure on $\R^n$ for the rest of the article.\\		
		
		For $\mu\in \Val_k(\R^n,F)$, \autoref{theorem:McMullenHomDecomp} allows us to define the polarization $\bar{\mu}:\K(\R^n)^k\rightarrow F$ by
		\begin{align*}
			\bar{\mu}(K_1,\dots,K_k)=\frac{1}{k!}\frac{\partial^k}{\partial\lambda_1\dots\partial\lambda_k}\Big|_0\mu\left(\sum_{j=1}^k\lambda_jK_j\right).
		\end{align*}
		Then $\bar{\mu}$ is in particular Minkowski additive in each argument. More explicitly, $\bar{\mu}$ can be expressed in the following way:
		\begin{align*}
			\bar{\mu}(K_1,\dots,K_k)=\frac{1}{k!}\sum_{j=1}^k(-1)^{j+k}\sum_{i_1<\dots<i_j}\mu\left(K_{i_1}+\dots+K_{i_j}\right).
		\end{align*}
		If $|\cdot|_F$ is a semi-norm on $F$, this implies for any $L\in\K(\R^n)$
		\begin{align}\label{eq:estiamtePolarization}
			\sup_{K_1,\dots,K_j\subset L}|\bar{\mu}(K_1,\dots,K_k)|_F\le C_k\sup_{K\subset L}|\mu(K)|_F
		\end{align}
		for a constant $C_k$ independent of the semi-norm $|\cdot|_F$, $L$, and $\mu$.\\
		
		If $\phi\in C^2(S^{n-1})$, then it is easy to see that the function $\|\phi\|_{C^2(S^{n-1})}h_{B_1(0)}+\phi$ is the support function of a convex body $K_\phi\in\mathcal{K}(\R^n)$, compare \cite{SchneiderConvexbodiesBrunn2014}*{Lemma 1.7.8}, i.e. $\phi$ is a difference of two support functions of convex bodies contained in a ball with radius $2\|\phi\|_{C^2(S^{n-1})}$. We extend $\bar{\mu}$ to a multilinear function on differences of support functions as follows: For $\phi_j=h_{K_j}-h_{L_j}$ for convex bodies $K_j,L_j\in \K(\R^n)$, $1\le j\le k$, we set
		\begin{align}\label{eq:formulaMultlinExtensionPolarization}
			\tilde{\mu}(\phi_1,\dots,\phi_k)=&\sum_{j=0}^k\frac{(-1)^{k+j}}{k!(k-l)!}\sum_{\sigma\in S_k}\bar{\mu}(K_{\sigma(1)},\dots,K_{\sigma_{j}},L_{\sigma(j+1)},\dots,L_{\sigma(k)}),
		\end{align}
		where $S_k$ denotes the symmetric group of $k$ elements. If $F$ is a locally convex vector space, then $\tilde{\mu}:C^2(S^{n-1})^k\rightarrow F$ extends to a distribution with values in the completion $\bar{F}$ of $F$. This was first shown for real-valued valuations by Goodey and Weil \cite{GoodeyWeilDistributionsvaluations1984} (we also refer to \cite{AleskerIntroductiontheoryvaluations2018}*{Section~7}), but extends to arbitrary locally convex vector spaces.
		\begin{theorem}\label{theorem:ExistenceGW}
			Let $F$ be a locally convex vector space. For every $\mu\in\Val_k(\R^n,F)$ there exists a unique distribution $\GW(\mu):C^\infty((S^{n-1})^k)\rightarrow\bar{F}$ such that
			\begin{align*}
				\GW(\mu)[h_{K_1}\otimes\dots\otimes h_{K_k}]=\bar{\mu}(K_1,\dots,K_k),
			\end{align*}
			for all smooth convex bodies $K_1,\dots,K_k\in\K(\R^n)$ with strictly positive curvature. 
		\end{theorem}
		
		\begin{remark}
			This implies in particular that $\GW(\mu)[\phi_1\otimes\dots\otimes \phi_k]\in F$ for all $\phi_1,\dots,\phi_k\in C^\infty(S^{n-1})$. 
		\end{remark}
		The main ingredient in the proof of \autoref{theorem:ExistenceGW} is the Schwartz Kernel Theorem (compare \cite{GaskproofSchwartzskernel1961}) combined with the following estimate, which follows directly from Eq.~\eqref{eq:estiamtePolarization} and Eq.~\eqref{eq:formulaMultlinExtensionPolarization} together with the observation that any $\phi\in C^2(S^{n-1})$ is a difference of support functions of convex bodies contained in the ball of radius $2\|\phi\|_{C^2(S^{n-1})}$.
		\begin{corollary}\label{corollary:estimateGW}
			There exists a constant $C>0$ such that for every $\mu\in\Val_k(\R^n,F)$ and every semi-norm $|\cdot|_F$ on $F$,
			\begin{align*}
				|\GW(\mu)[\phi_1\otimes\dots\otimes\phi_k]|_F\le  C\|\mu\|_F \prod_{j=1}^k \|\phi_j\|_{C^2(S^{n-1})}. 
			\end{align*}
		\end{corollary}
		The valuation property has the following consequence for the support of these distributions.
		\begin{proposition}[\cite{AleskerP.McMullensconjecture2000}*{Proposition 3.3}]
			\label{proposition:SupportValuations}
			For $\mu\in\Val_k(\R^n,F)$, the support of $\GW(\mu)$ is contained in the diagonal in $(S^{n-1})^k$.
		\end{proposition}
		
		The following notion was introduced in \cite{Knoerrsupportduallyepi2021}*{Definition~6.1}. For $\mu\in\Val(\R^n,F)$, let $\mu=\sum_{k=0}^n\mu_k$ be its homogeneous decomposition and define its \emph{vertical support} to be the set
		\begin{align*}
			\vsupp\mu:=\bigcup_{k=1}^n \Delta_k^{-1}(\supp\GW(\mu_k)),
		\end{align*}
		where $\Delta_k:S^{n-1}\rightarrow (S^{n-1})^k$ denotes the diagonal embedding. This set admits the following more intrinsic characterization.		
		\begin{proposition}[\cite{Knoerrsupportduallyepi2021}*{Proposition 6.14}]
			\label{proposition:characterization_vsupp}
			Let $\mu\in\Val(\R^n,F)$. The vertical support is minimal (with respect to inclusion) among all closed sets $A\subset S^{n-1}$ with the following property: If $K,L\in\K(\R^n)$ are two convex bodies with $h_K=h_L$ on a neighborhood of $A$, then $\mu(K)=\mu(L)$.
		\end{proposition}
		Note that the vertical support of a $0$-homogeneous valuation is empty by definition. For continuous area measures, this leads to a counter intuitive notion of support, and we introduce a variation of this notion in \autoref{section:HomDecompArea}.
		
		The following characterization was originally obtained by Hadwiger \cite{HadwigerVorlesungenuberInhalt1957} for real-valued valuations, but the argument extends to the following more general case (compare also \cite{McMullenWeaklycontinuousvaluations1983}).
		\begin{theorem}
			\label{theorem:HadwigerVolume}
			Let $F$ be a Hausdorff topological vector space. If $\mu\in \Val_n(\R^n,F)$, then there exists $\nu\in F$ with $\mu(K)=\vol(K)\nu$ for every $K\in\K(\R^n)$.
		\end{theorem}
		
		\subsection{Contact geometry on $\R^n\times S^{n-1}$}
		We refer to \cite{CannasdaSilvaLecturessymplecticgeometry2001} for an introduction to contact and symplectic geometry.\\
		
		Let us denote by $\pi_1$ and $\pi_2$ the natural projections from $\R^n\times S^{n-1}$ onto $\R^n$ and $S^{n-1}$ respectively. We have a natural $1$-form $\alpha$ on $\R^n\times S^{n-1}$ given by
		\begin{align*}
			\alpha|_{(x,v)}=\langle v,d\pi_1(\cdot)\rangle\quad \text{in}~(x,v)\in\R^n\times S^{n-1}.
		\end{align*}
		This equips $\R^n\times S^{n-1}$ with the structure of a contact manifold, i.e. the restriction of the form $d\alpha$ to the contact distribution $H$ given by
		\begin{align*}
			H_{(x,v)}:=\ker \alpha|_{(x,v)}\quad \text{in}~(x,v)\in\R^n\times S^{n-1}
		\end{align*}
		is nondegenerate, and consequently defines a symplectic form. This implies that there exists a unique smooth vector field $T$, called the Reeb vector field, such that
		\begin{align*}
			&i_T\alpha=1, &&i_Td\alpha=0.
		\end{align*}
		In particular, we have the decomposition $T_{(x,v)}(\R^n\times S^{n-1})^*\cong \R T|_{(x,v)}\oplus \ker\alpha|_{(x,v)}$ for all $(x,v)\in\R^n\times S^{n-1}$. Note that this implies that any form $\omega$ that satisfies $\alpha\wedge \omega=0$ is a multiple of $\alpha$, since in this case $0=i_T(\alpha\wedge\omega)=\omega-\alpha\wedge i_T\omega$.\\
		We require the Lefschetz decomposition for forms on symplectic vector spaces.		
			\begin{proposition}[\cite{HuybrechtsComplexgeometry2005}*{Proposition~1.2.30}]
				\label{proposition:LefschetzDecomposition}
				Let $(W,\omega_s)$ be a symplectic vector space of dimension $2n$ and let $L:\Lambda^*W^*\rightarrow \Lambda^*W^*$, $\tau\mapsto \omega_s\wedge\tau$ be the Lefschetz operator. For $0\le k\le n$ let $P^k:=\{\tau\in\Lambda^kW^*:L^{n-k+1}\tau=0\}$ denote the space of primitive forms. Then the following holds:
				\begin{enumerate}
					\item There exists a direct sum decomposition $\Lambda^kW^*=\bigoplus_{i\ge 0}L^{i}P^{k-2i}$.
					\item $L^{n-k}:\Lambda^kW^*\rightarrow\Lambda^{2n-k}W^*$ is an isomorphism.
				\end{enumerate}
			\end{proposition}
		
		In our case, the symplectic vector space is of a specific form: By construction, 
		\begin{align*}
			H_{(x,v)}=\ker \alpha|_{(x,v)}=v^\perp\oplus v^\perp\quad \text{in}~(x,v)\in\R^n\times S^{n-1},
		\end{align*}
		and it is easy to check that $d\alpha$ is given by
		\begin{align}
			\label{eq:dAlphaContactPlane}
			d\alpha|_{H_{(x,v)}}((\xi_1,v_1),(\xi_2,v_2))=\langle v_1,\xi_2\rangle-\langle v_2,\xi_1\rangle\quad \text{for}~\xi_1,\xi_2,v_1,v_2\in v^\perp.
		\end{align}
		Recall that for a real vector space $W$, $W\times W^*$ is naturally a symplectic vector space with respect to the symplectic form
		\begin{align*}
			\omega_s((w_1,\eta_1),w_2,\eta_2)=\eta_1(w_2)-\eta_2(w_1),\quad w_1,w_2\in W, \eta_1,\eta_2\in W^*.
		\end{align*}
		If $W$ carries a Euclidean structure, the induced isomorphism $W\cong W^*$ equips $W\times W$ with a natural symplectic form. For $W=v^\perp$, this is precisely the form given in Eq.~\eqref{eq:dAlphaContactPlane}. Let us denote by $\Lambda^{k,l}(W\times W^*)^*$ and $\Lambda^{k,l}(W\times W)^*$ the space of forms of bidegree $(k,l)$ induced by the product structure and let $P\Lambda^{k,l}(W\times W^*)^*$ and $P\Lambda^{k,l}(W\times W)^*$ the subspace of primitive $(k,l)$-forms with respect to the respective symplectic structures. Note that $W\times W^*$ carries an operation of $\GL(W)$ via the diagonal action.
			\begin{proposition}[\cite{KnoerrMongeAmpereoperators2024}*{Corollary~6.8}]\label{proposition:symplecticIrreducible}
				Let $W$ be an $n$-dimensional real vector space. The space $P\Lambda^{k,n-k}(W\times W^*)$ is an irreducible representation of $\GL(n,\R)$.
			\end{proposition}
			\begin{remark}
				This result is stated in \cite{KnoerrMongeAmpereoperators2024} for $W=\R^n$ and a complex linear extension of this action to a representation of $\GL(n,\C)$. Since the action is algebraic, it is already determined by the action of $\GL(n,\R)$, so this defines an irreducible representation of $\GL(n,\R)$ (compare also the proof of \cite{KnoerrMongeAmpereoperators2024}*{Theorem 1.3}). 
			\end{remark}
			
			In our case, we have an action of $\GL(n,\R)$ on $\R^n\times S^{n-1}$ by $G_g(x,v)=(gx,\frac{g^{-T}v}{|g^{-T}v|})$ for $g\in\GL(n,\R)$. In particular, for $g\in \GL(n,\R)$,
			\begin{align}\label{eq:groupActionAlpha}
				G_g^*\alpha|_{(x,v)}=|g^{-T}v|^{-1}\cdot \alpha|_{(x,v)}\quad\text{for}~(x,v)\in\R^n\times S^{n-1},
			\end{align}			
			so $\GL(n,\R)$ operates on $\R^n\times S^{n-1}$ by contactomorphisms. Moreover,
			\begin{align*}
				G_g^*d\alpha|_{(x,v)}=|g^{-T}v|^{-1}\cdot d\alpha|_{(x,v)}-\frac{1}{|g^{-T}v|^2}\langle g^{-T}v,\pi_2(\cdot)\rangle \wedge \alpha.
			\end{align*}
			Let $\GL(n,\R)_v\subset\GL(n,\R)$ be the stabilizer of $v\in S^{n-1}$, i.e. the subgroup of all $g\in\GL(n,\R)$ such that $\frac{g^{-T}v}{|g^{-T}v|}=v$. Then $G_g^*d\alpha|_{(x,v)}=|g^{-T}v|^{-1}\cdot d\alpha$ since the form $\langle \frac{g^{-T}v}{|g^{-T}v|},\pi_2(\cdot)\rangle=\langle v,\pi_2(\cdot)\rangle$ vanishes on $T_{(x,v)}(\R^n\times S^{n-1})\cong \R^n\oplus v^\perp$.\\
			
			We will use this action to obtain a suitable representation on a space of differential form corresponding to the space in \autoref{proposition:symplecticIrreducible}. Let us consider the bundle $J^{k,n-k}$ over $\R^n\times S^{n-1}$ with fibers
			\begin{align*}
				J^{k,n-k}|_{(x,v)}:=\{\tau\in \Lambda^k(T_{x}\R^n)^*\otimes \Lambda^{n-k}(T_v S^{n-1})^*:\alpha\wedge\tau=d\alpha\wedge\tau=0\}.
			\end{align*}
			Then it is easy to check that this defines a $\GL(n,\R)$-equivariant bundle. In particular, $J^{k,n-k}|_{(0,v)}$ is closed under the induced action of the subgroup $\GL(n,\R)_v$ of $\GL(n,\R)$.
			\begin{lemma}\label{lemma:Kernel_J_to_Area}
				Let $\tau\in J^{k,n-k}|_{(x,v)}$. If $i_T\tau$ belongs to the ideal generated by $\alpha$ and $d\alpha$, then $\tau=0$.
			\end{lemma}
			\begin{proof}
				Since $\alpha\wedge\tau=0$, we have $0=i_T(\alpha\wedge \tau)=\tau-\alpha\wedge i_T\tau$. If $i_T\tau=\alpha\wedge \omega_1+d\alpha\wedge \omega_2$, we obtain the following for the restriction to the contact distribution $H$:
				\begin{align*}
					0=i_T(d\alpha\wedge \tau)|_H =d\alpha\wedge  i_T\tau|_H=d\alpha^2\wedge\omega_2|_H.
				\end{align*}
				Due to the Lefschetz decomposition in \autoref{proposition:LefschetzDecomposition}, this implies $\omega_2|_H=0$, so $\alpha\wedge \omega_2=0$, which implies that $\omega_2$ is a multiple of $\alpha$. Thus $i_T\tau$ is a multiple of $\alpha$, which shows $\tau=\alpha\wedge i_T\tau=0$.
			\end{proof}
		
			We will be mainly interested in smooth sections of the bundle defined above, which corresponds to the space
			\begin{align*}
				\mathcal{J}^{k,n-k}(\R^n)^{tr}:=\{\omega\in \Omega^{k,n-k}(\R^n\times S^{n-1})^{tr}:\alpha\wedge\omega=d\alpha\wedge\omega =0\}.
			\end{align*}
			These spaces have been previously considered in \cite{BernigDualareameasures2019}. The following is a version of the Lefschetz decomposition for forms of degree $n-1$ on $\R^n\times S^{n-1}$.
				\begin{lemma}\label{lemma:primitivityContactPlane}
				If $\omega\in \Omega^{n-1}(\R^n\times S^{n-1})^{tr}$, then there exists a form $\xi\in \Omega^{n-3}(\R^n\times S^{n-1})^{tr}$ such that the restriction of $\omega-d\alpha\wedge\xi$ to the contact distribution is primitive, i.e. such that
				\begin{align*}
					d\alpha\wedge(\omega-d\alpha\wedge\xi)=\alpha\wedge\zeta
				\end{align*}
				for some $\zeta\in \Omega^n(\R^n\times S^{n-1})^{tr}$.
			\end{lemma}	
			\begin{proof}
				By the Lefschetz decomposition applied pointwise to the contact distribution $H$, we find a smooth differential form $\xi\in \Omega^{n-3}(\R^n\times S^{n-1})^{tr}$ such that $d\alpha\wedge(\omega-d\alpha\wedge\xi)|_H=0$. In particular, 
				\begin{align*}
					\alpha\wedge d\alpha\wedge(\omega-d\alpha\wedge\xi)=0,
				\end{align*}
				Thus, $d\alpha\wedge(\omega-d\alpha\wedge\xi)$ is a multiple of $\alpha$, i.e. the claim follows by setting $\zeta:=i_T(d\alpha\wedge(\omega-d\alpha\wedge\xi))$. 
			\end{proof}
			
			Let $I_{k,n-1-k}=\langle \alpha,d\alpha\rangle\cap \Omega^{k,n-1-k}(\R^n\times S^{n-1})^{tr}$, where $\langle \alpha,d\alpha\rangle$ denotes the ideal generated by $\alpha$ and $d\alpha$.			\begin{corollary}\label{corollary:RelationJtoKernel}
				The maps
				\begin{align*}
					\mathcal{J}^{k+1,n-1-k}(\R^n)^{tr}&\rightarrow \Omega^{k,n-1-k}(\R^n\times S^{n-1})^{tr}/I_{k,n-1-k}\\
					\tau&\mapsto [i_T\tau],\\
					J^{k+1,n-1-k}|_{(0,v)}&\rightarrow P\Lambda^{k,n-1-k}(v^\perp\times v^\perp)^*\\
					\tau&\mapsto i_T\tau
				\end{align*}
				are isomorphisms.
			\end{corollary}
			\begin{proof}
				We will only consider the first map, the argument for the second is similar. If $[i_T\tau]=0$, then $i_T\tau$ is contained in the ideal generated by $\alpha$ and $d\alpha$, and so $\tau=0$ by \autoref{lemma:Kernel_J_to_Area}. If $\omega\in \Omega^{k,n-1-k}(\R^n\times S^{n-1})^{tr}$, then there exist differential forms $\xi_1\in \Omega^{k-1,n-2-k}(\R^n\times S^{n-1})^{tr}$, $\xi_2\in \Omega^{k,n-k}(\R^n\times S^{n-1})^{tr}$ such that 
				\begin{align*}
					d\alpha\wedge(\omega-d\alpha\wedge\xi_1)=\alpha\wedge \xi_2.
				\end{align*}
				by \autoref{lemma:primitivityContactPlane}. Thus the form $\tau:=\alpha\wedge(\omega-d\alpha\wedge\xi_1)$ satisfies $d\alpha\wedge\tau=0$ and $\alpha\wedge\tau=0$, so $\tau\in\mathcal{J}^{k+1,n-1-k}(\R^n)^{tr}$ and
				\begin{align*}
					i_T\tau=\omega-d\alpha\wedge\xi_1-\alpha\wedge(i_T\omega-d\alpha\wedge i_T\xi_1),
				\end{align*}
				i.e. $i_T\tau-\omega\in I_{k,n-1-k}$. The claim follows. 
			\end{proof}

			The next result is a contact geometric analog of \autoref{proposition:symplecticIrreducible}.
			\begin{proposition}\label{proposition:JIrreducibleStablizer}
				For $v\in S^{n-1}$ let $G_v$ denote the subgroup of $\GL(n,\R)$ of all elements that satisfy $g^{-T}v=v$  and that preserve the decomposition $\R^n=\R v\oplus v^\perp$. Then $J^{k,n-k}|_{(0,v)}$ is an irreducible representation of $G_v\cong \GL(v^\perp)$.
			\end{proposition}
			\begin{proof}
				Note that $\alpha|_{(0,v)}$ is invariant under the action of $\GL(n,\R)_v$ due to Eq.~\eqref{eq:groupActionAlpha}, so it is in particular invariant under $G_v$. In addition, the vector $T|_{(0,v)}$ is invariant under $G_v$ since this group preserves the decomposition $\R^n=\R v\oplus v^\perp$ and fixes $v$. In particular, 
				\begin{align*}
						J^{k+1,n-1-k}|_{(0,v)}&\rightarrow P\Lambda^{k,n-1-k}(v^\perp\times v^\perp)^*\\
					\tau&\mapsto i_T\tau
				\end{align*}
				is bijective and commutes with the action of $G_v$ on both spaces. We thus need to show that $P\Lambda^{k,n-1-k}(v^\perp\times v^\perp)^*$ is an irreducible representation of $G_v$. Let $W=v^\perp$ with the natural representation of $\GL(W)$. We claim that $P\Lambda^{k,n-1-k}(v^\perp\times v^\perp)^*$ is isomorphic to $P\Lambda^{k,n-1-k}(W\times W^*)^*$ as a representation of $\GL(W)\cong G_v$. In order to see this, note that for $g\in G_v$, the action induced on $H_{(0,v)}\cong v^\perp\oplus v^\perp$ is given by
				\begin{align*}
					(w_1,w_2)\mapsto dG_g (w_1,w_2)=&\frac{d}{dt}\Big|_0 \left(g(0+tw_1),\frac{g^{-T}(v+tw_2)}{|g^{-T}(v+tw_2)|}\right)\\
					=&(gw_1,g^{-T}w_2),
				\end{align*}
				where we used that $\frac{d}{dt}|_0|g^{-T}(v+tw_2)|=\langle g^{-T}v,g^{-T}w_2\rangle=0$ if $g$ preserves the decomposition $\R v\oplus v^\perp$. Thus  $P\Lambda^{k,n-1-k}(v^\perp\times v^\perp)^*$ is isomorphic to $P\Lambda^{k,n-1-k}(W\times W^*)^*$ as a representation of $G_v\cong \GL(W)$. Since the latter is irreducible under the action of $\GL(W)$ due to \autoref{proposition:symplecticIrreducible}, the claim follows.
			\end{proof}
			
			\begin{theorem}\label{theorem:IrreducibilityDiffForms}
			For $1\le k\le n$, let $A\subset \mathcal{J}^{k,n-k}(\R^n)^{tr}$ be a nontrivial $\GL(n,\R)$-invariant $C^\infty(S^{n-1})$-submodule. Then $A=\mathcal{J}^{k,n-k}(\R^n)^{tr}$.
		\end{theorem}
		\begin{proof}
			By construction, we have a well defined and surjective map
			\begin{align*}
				\mathcal{J}^{k,n-k}(\R^n)^{tr}&\rightarrow J^{k,n-k}|_{(0,v)}\\
				\tau&\mapsto \tau|_{(0,v)}
			\end{align*}
			which commutes with the action of $\GL(n,\R)_v$ on both spaces. Since $J^{k,n-k}|_{(0,v)}$ is an irreducible representation of $\GL(n,\R)_v$ by \autoref{proposition:JIrreducibleStablizer}, the image of $A$ under this map is either $0$ or coincides with $J^{k,n-k}|_{(0,v)}$. Let $A_v$ denote the image of $A$ under this map. Since we are only considering translation invariant forms and $\GL(n,\R)$ operates transitively on $S^{n-1}$, it is easy to check that $A$ is trivial if $A_v=0$ for some $v\in S^{n-1}$. Thus this map is surjective for all $v\in S^{n-1}$. In particular, for any point $v\in S^{n-1}$ there exist elements $\omega^j_v\in A$, $1\le j\le \dim J^{k,n-k}|_{(0,v)}$ such that their images under this map form a basis of $J^{k,n-k}|_{(0,v)}$. Using a local trivialization, we obtain a neighborhood $U_v$ of $v$ such that the images of $\omega^j_v$, $1\le j\le \dim J^{k,n-k}|_{(0,v)}$, form a basis for each point in this neighborhood. If $\tau\in \mathcal{J}^{k,n-k}(\R^n)^{tr}$ is an arbitrary element, choose a partition of unity subordinate to the cover $U_v$ of $S^{n-1}$. Since $S^{n-1}$ is compact, we thus obtain $v_1\dots,v_N\in S^{n-1}$ and $\phi_i\in C^\infty_c(U_{v_i})$ such that $\sum_{i=1}^N\phi_i=1$. For each $1\le i\le N$, we obtain smooth sections $\eta^j_i\in C^\infty_c(U_i)$ such that $\sum_{j}\eta^j_i\omega^j_{v_i}|_{(0,v)}=\phi_i\tau|_{(0,v)}$ for all $v\in U_i$. Since $\tau$ is translation invariant, this implies
			\begin{align*}
				\tau=\sum_{i=1}^N\phi_i\tau=\sum_{i=1}^N\sum_{j}\eta^j_i\omega^j_{v_i}.
			\end{align*}
			Because $\omega^j_{v_i}\in A$ by assumption and as $A$ is a $C^\infty(S^{n-1})$ module, this implies $\tau\in A$.
		\end{proof}

		\subsection{Construction of area measures with the normal cycle}
	\label{section:ConstructionSmoothArea}
		In this section, we discuss a construction of continuous area measures from \cite{WannererIntegralgeometryunitary2014,Wannerermoduleunitarilyinvariant2014}, which is based on the properties of the normal cycle. We refer to \cite{AleskerFuTheoryvaluationsmanifolds.2008}*{Section 2} for a general background on the normal cycle and only collect the properties we need in this article. As a set, the normal cycle of $K\in\mathcal{K}(\R^n)$ is given by	
		\begin{align}
			\label{eq:DefNormalCycle}
			\nc(K)=\{(x,v)\in\R^n\times S^{n-1}:x\in\partial K, \langle v,x-y\rangle\ge 0~\text{for all}~y\in K\}.
		\end{align}
		As discussed in the introduction, it is naturally an $(n-1)$-dimensional Lipschitz submanifold, and any orientation on $\R^n$ induces a natural orientation of $\nc(K)$, which we fix to be the standard orientation of $\R^n$ for simplicity. For any $\omega\in\Omega^{n-1}(\R^n\times S^{n-1})$, we obtain a map $\Phi_\omega:\K(\R^n)\rightarrow \M(S^{n-1})$ by setting
		\begin{align*}
			\Phi_\omega(K;\phi)=\int_{\nc(K)}\pi_2^*\phi \wedge\omega\quad\text{for}~\phi\in C(S^{n-1}),
		\end{align*}
		where $\pi_2:\R^n\times S^{n-1}\rightarrow S^{n-1}$ denotes the natural projection. Note that \autoref{corollary:RelationSubdiffNormalCycle} implies that $(x,v)\in \nc(K)$ if and only if $x\in\partial h_K(v)$.  Since the subdifferentials of $h_K$ and $h_L$ agree on any open set $U\subset S^{n-1}$ such that $h_K|_U=h_L|_U$ (compare the discussion in \autoref{section:preliminaries_convexbodies}), we directly obtain the following.		
		\begin{corollary}
			For every $\omega\in \Omega^{n-1}(\R^n\times S^{n-1})$, the map $\Phi_\omega:\K(\R^n)\rightarrow\M(S^{n-1})$ is locally determined.
		\end{corollary}
		
		It follows from \cite{WannererIntegralgeometryunitary2014}*{Lemma 2.4.} that $\Phi_\omega:\mathcal{K}(\R^n)\rightarrow \M(S^{n-1})$ is continuous in the weak* topology, which relies on the fact that the mass of the normal cycle along a convergent sequence of convex bodies is bounded. We require a quantitative estimate for this bound, which we will obtain from certain uniform estimates of $\Phi_\omega$ in terms of the surface area measures $S_k$, $0\le k\le n-1$.\\
		
	Consider the Lipschitz--Killing forms $\kappa_k$ defined for $0\le k\le n-1$ by
	\begin{align*}
		\kappa_k:=\frac{1}{k!(n-1-k)!}\sum_{\sigma\in S_{n}}\sign(\sigma)v_{\sigma(1)}dx_{\sigma(2)}\dots dx_{\sigma(k+1)}\wedge dv_{\sigma(k+2)}\dots dv_{\sigma(n)},
	\end{align*}
	where $x_j$, $1\le j\le n$, denote orthonormal coordinates on $\R^n$ and $v_j$, $1\le j\le n$, the induced coordinate functions on $S^{n-1}$. Then
	\begin{align*}
		S_k(K)=&\frac{1}{(n-k)\omega_{n-k}}\Phi_{\kappa_k}(K),&&V_k(K)=S_k(K;S^{n-1})
	\end{align*} 
	are the  $k$th intermediate area measure and $k$th intrinsic volume, respectively, compare \cite{FuAlgebraicintegralgeometry2014}*{Section~2.1}. We equip $\Lambda^{k,l}(T_{(x,v)}\R^n\times S^{n-1})^*$ with the operator norm induced from the inner product on $\R^n$ and we denote this norm by $|\cdot|$. Let $\Omega^{k,l}(\R^n\times S^{n-1})$ denote the space of $(k,l)$-forms on $\R^n\times S^{n-1}$.
	\begin{proposition}\label{proposition:boundSmoothArea}
		For $0\le k\le n-1$ and any $\omega\in \Omega^{k,n-1-k}(\R^n\times S^{n-1})$, we have for $K\in\mathcal{K}(\R^n)$ and $\phi\in C(S^{n-1})$,
		\begin{align}
			\label{eq:estimateTV}
			|\Phi_\omega(K;\phi)|\le (n-k)\omega_{n-k}\sup_{x\in K,v\in S^{n-1}}|\omega|_{(x,v)}| \int_{S^{n-1}} |\phi| dS_k(K).
		\end{align}
		In particular, if $\omega$ is translation invariant, then  
		\begin{align*}
			|\Phi_\omega(K;\phi)|\le (n-k)\omega_{n-k} \|\omega\|_\infty\|\phi\|_\infty V_k(K).
		\end{align*}
	\end{proposition}
	\begin{proof}
		If $K$ is a smooth convex body with strictly positive Gauss curvature, then $\nc(K)$ is the graph of the map 
		\begin{align*}
			G_K:S^{n-1}&\rightarrow \nc(K)\\
			v&\mapsto (\nu_K^{-1}(v),v),
		\end{align*}
		where $\nu_K:\partial K\rightarrow S^{n-1}$ denotes the Gauss map, and 
		\begin{align*}
			\Phi_\omega(K;\phi)=\int_{S^{n-1}} \phi(v)G_K^*\omega.
		\end{align*}
		Since the differential of $\nu_K^{-1}$ is self-adjoint and positive definite, we find an orthonormal basis $e_1,\dots,e_{n-1}$ of $v^\perp$ and $\lambda_1,\dots,\lambda_{n-1}>0$ such that the tangent space of $\nc(K)$ in $(\nu_K^{-1}(v),v)$ is spanned by the vectors 
		\begin{align*}
			(\lambda_1 e_1,e_1),\dots,(\lambda_{n-1} e_{n-1},e_{n-1})\in T_{(\nu_K^{-1}(v),v)}(\R^n\times S^{n-1})\cong\R^n\times v^\perp.
		\end{align*}  
		Without loss of generality, we may assume that $v,e_1,\dots,e_{n-1}$ forms an oriented orthonormal basis of $\R^n$. Then the pullback of $\omega$ to $S^{n-1}$ by $G_K$ is given in $v\in S^{n-1}$ by
		\begin{align*}
			\omega|_{(\nu^{-1}_K(v),v)}((\lambda_1 e_1,e_1),\dots,(\lambda_{n-1} e_{n-1},e_{n-1}))\vol_{S^{n-1}}.
		\end{align*}
		If we set $\bar{e}_j=(e_j,0)$ and $\bar{f}_j=(0,e_j)$, we have, as $\omega$ is of bidgree $(k,n-1-k)$,
		\begin{align*}
			&\omega|_{(\nu^{-1}_K(v),v)}((\lambda_1 e_1,e_1),\dots,(\lambda_{n-1} e_{n-1},e_{n-1}))\\
			=&\omega|_{(\nu^{-1}_K(v),v)}(\lambda_1 \bar{e}_1+\bar{f}_1,\dots,\lambda_1 \bar{e}_{n-1}+\bar{f}_{n-1})\\
			=&\frac{1}{k!(n-1-k)!}\sum_{\pi\in S_{n-1}}\sign(\pi)\lambda_{\pi_{1}}\dots\lambda_{\pi_k}\omega|_{(\nu^{-1}_K(v),v)}(\bar{e}_{\pi_1},\dots,\bar{e}_{\pi_k},\bar{f}_{\pi_{k+1}},\dots,\bar{f}_{\pi_{n-1}}).
		\end{align*}
		Since
		\begin{align*}
			|\sign(\pi)\omega|_{(\nu^{-1}_K(v),v)}(\bar{e}_{\pi_1},\dots,\bar{e}_{\pi_k},\bar{f}_{\pi_{k+1}},\dots,\bar{f}_{\pi_{n-1}})|\le|\omega|_{(\nu^{-1}_K(v),v)}|
		\end{align*}
		and
		\begin{align*}
			\kappa_k|_{(\nu^{-1}_K(v),v)}(\bar{e}_{\pi_1},\dots,\bar{e}_{\pi_k},\bar{f}_{\pi_{k+1}},\dots,\bar{f}_{\pi_{n-1}})=\sign(\pi),
		\end{align*}
		we obtain, using that $\kappa_k$ is also of bidegree $(k,n-1-k)$,
		\begin{align*}
			&|\omega|_{(\nu^{-1}_K(v),v)}((\lambda_1 e_1,e_1),\dots,(\lambda_{n-1} e_{n-1},e_{n-1}))|\\
			\le &|\omega|_{(\nu^{-1}_K(v),v)}|\cdot \kappa_k|_{(\nu^{-1}_K(v),v)}((\lambda_1 e_1,e_1),\dots,(\lambda_{n-1} e_{n-1},e_{n-1})).
		\end{align*}
		Using this estimate, it is now easy to see that Eq.~\eqref{eq:estimateTV} holds for every smooth convex body and every $\phi\in C(S^{n-1})$. Since both sides are continuous in $K$ for $\phi\in C(S^{n-1})$ fixed, the inequality holds for all convex bodies, which shows the claim.
	\end{proof}
	Recall that the mass of an $m$-current $T$ on a Riemannian manifold $U$ is defined by
	\begin{align*}
		M(T)=\sup\{|T(\omega)|:\omega\in \Omega_c^m(U),\|\omega\|_\infty\le 1\}.
	\end{align*}
	The previous result implies the following estimate for the mass of the normal cycle in terms of the intrinsic volumes.
		\begin{corollary}
		For $K\in\mathcal{K}(\R^n)$, $M(\nc(K))\le \sum_{j=0}^{n-1}(n-j)\omega_{n-j}V_j(K)$.
	\end{corollary}
		
	The space of differential forms that vanish on all normal cycles was described by Bernig and Bröcker in \cite{BernigBroeckerValuationsmanifoldsRumin2007}. The following result is essentially contained in the proof of  \cite{BernigBroeckerValuationsmanifoldsRumin2007}*{Theorem 1}, however, since it is not stated explicitly, we present the relevant part of the proof for the convenience of the reader.  
	\begin{lemma}[\cite{BernigBroeckerValuationsmanifoldsRumin2007}]
		\label{lemma:KernelCNC}
		Let $\omega\in \Omega^{n-1}(\R^n\times S^{n-1})$. If $\Phi_\omega(K)=0$ for all $K\in\mathcal{K}(\R^n)$, then $\omega$ belongs to the ideal generated by $\alpha$ and $d\alpha$.
	\end{lemma}
	\begin{proof}
		Note first that the normal cycle of a convex body vanishes on multiples of $\alpha$ and $d\alpha$.	Let $\omega\in \Omega^{n-1}(\R^n\times S^{n-1})$ be a form such that $\Phi_\omega(K)=0$ for all $K\in\mathcal{K}(\R^n)$. If $K$ is a smooth convex body with strictly positive curvature, then, as in the proof of \autoref{proposition:boundSmoothArea}, the tangent spaces of $\nc(K)$ are of the form
		\begin{align*}
			T_{(x,v)}\nc(K)=\mathrm{span}(\lambda_1 e_1,e_1),\dots,(\lambda_{n-1}e_{n-1},e_{n-1})\subset v^\perp\oplus v^\perp,
		\end{align*}
		where $\lambda_1,\dots,\lambda_{n-1}>0$ and $e_1,\dots e_{n-1}$ are an orthonormal basis of $v^\perp$. Let us call subspaces of $T_{(x,v)}(\R^n\times S^{n-1})\cong \R^n\times v^\perp$ of this form strictly positive isotropic subspaces. It is not difficult to see that for every such subspace $E$ there is a smooth convex body $K$ with strictly positive curvature such that $E= T_{(x,v)}\nc(K)$. In particular, if $\Phi_\omega$ vanishes on all smooth convex bodies, this implies that the restriction of $\omega$ to all strictly positive isotropic subspaces vanishes. Note that these subspaces belong to the contact distribution. It now follows from \cite{AbardiaEvequozEtAlFlagareameasures2019}*{Lemma 2.4} that the restriction of $\omega$ to the contact distribution is a multiple of $d\alpha$. After subtracting a multiple of $d\alpha$, we may thus assume that the restriction of $\omega$ to the contact distribution vanishes. This implies that $\alpha\wedge \omega=0$, and consequently, $\omega=\alpha\wedge i_T\omega$ is a multiple of $\alpha$.
	\end{proof}
	
	If $\omega\in \Omega^{n-1}(\R^n\times S^{n-1})^{tr}$ is translation invariant, then $\Phi_\omega$ is also translation invariant, so we obtain a well defined element in $\Area(\R^n)$. However, the map $\omega\mapsto \Phi_\omega$ does not behave well with respect to the natural operation of $\GL(n,\R)$ on both spaces. In order to circumvent this problem, consider the space
	\begin{align*}
		\mathcal{J}^n(\R^n)^{tr}:=\{\tau\in \Omega^{n}(\R^n\times S^{n-1})^{tr}:\alpha\wedge\tau=d\alpha\wedge\tau =0\}.
	\end{align*}
	Using the notation from the previous section, we have
	\begin{align*}
		\mathcal{J}^n(\R^n)^{tr}=\bigoplus_{k=0}^{n-1}\mathcal{J}^{k+1,n-1-k}(\R^n)^{tr}.
	\end{align*}
	We are going use the map $\omega\mapsto \Phi_\omega$ to construct a $\GL(n,\R)$-equivariant map 
	\begin{align*}
		\mathcal{J}^{n}(\R^n)\rightarrow \Area(\R^n)
	\end{align*}
	as follows: If $\phi\in C^1(S^{n-1})$, we may consider its $1$-homogeneous extension $\tilde{\phi}\in C^1(\R^n\setminus\{0\})$. In particular, its gradient $\nabla\tilde{\phi}:\R^n\setminus\{0\}\rightarrow \R^n$ is $0$-homogeneous and thus induces a well-defined continuous map
	\begin{align*}
		\bar{d}\phi: S^{n-1}\rightarrow \R^n.
	\end{align*}
	Recall that the action of $g\in \GL(n,\R)$ on $\phi\in C(S^{n-1})$ is given by 
	\begin{align*}
		[g\cdot \phi](v)=|g^{T}v|\phi\left(\frac{g^{T}v}{|g^{T}v|}\right)\quad \text{for}~v\in S^{n-1}.
	\end{align*}
	In particular, $\bar{d}(g\cdot \phi)=g(\bar{d}\phi)$. Using the product structure of $\R^n\times S^{n-1}$, we consider $\bar{d}\phi$ as as a vector field $X_\phi$ on $\R^n\times S^{n-1}$. If we choose  orthonormal coordinates $(x_1,\dots,x_n)$ on $\R^n$ and induced coordinates $(v_1,\dots,v_n)$ on $S^{n-1}$, then 
	\begin{align*}
		X_\phi|_{(x,v)}=\sum_{j=1}^n \frac{\partial \tilde{\phi}(v)}{\partial v_j} \frac{\partial }{\partial x_j}\Big|_{(x,v)}.
	\end{align*}
	In particular, $i_{X_\phi}\alpha=\pi_2^*\phi$ on $\R^n\times S^{n-1}$ since $\tilde{\phi}$ is $1$-homogeneous. Given $\tau\in \mathcal{J}^n(\R^n)^{tr}$, we consider the map $K\mapsto \Phi_{i_T\tau}(K)$. Since $\tau\in \mathcal{J}^n(\R^n)^{tr}$, $\tau=\alpha\wedge i_T\tau$, so for $\phi\in C^1(S^{n-1})$, we have 
	\begin{align*}
		i_{X_\phi}\tau=i_{X_\phi}(\alpha\wedge i_T\tau)=\pi_2^*\phi\wedge i_T\tau-\alpha\wedge i_{X_\phi}i_T\tau.
	\end{align*}
	Since the normal cycle vanishes on multiples of $\alpha$, this shows
	\begin{align*}
		\Phi_{i_T\tau}(K;\phi)=\int_{\nc(K)}i_{X_\phi}\tau
	\end{align*}
	for all $\phi\in C^1(S^{n-1})$. Since $(G_{g^{-1}})_*\nc(K)=\sign(\det g)\nc(g^{-1}K)$ for $g\in \GL(n,\R)$, this implies
	\begin{align*}
		\Phi_{i_T\tau}(g^{-1}K;g^{-1}\cdot\phi)=&\int_{\nc(g^{-1}K)}i_{X_{g^{-1}\cdot\phi}}\tau=\sign(\det g)\int_{(G_{g^{-1}})_*\nc(K)}i_{X_{g^{-1}\cdot\phi}}\tau.
	\end{align*}
	 Since $\bar{d}(g\cdot \phi)=g(\bar{d}\phi)$, we have $X_{g^{-1}\cdot\phi}=(G_{g^{-1}})_*X_{\phi}$, and consequently,
		\begin{align*}
		\Phi_{i_T\tau}(g^{-1}K;g^{-1}\cdot\phi)=&\sign(\det g)\int_{\nc(K)}G_{g^{-1}}^*\left(i_{X_{g^{-1}\cdot\phi}}\tau\right)\\
		=&\sign(\det g)\int_{\nc(K)}i_{X_{\phi}}G_{g^{-1}}^*\tau\\
		=&\sign(\det g)\Phi_{i_T(G_{g^{-1}}^*\tau)}(K;\phi).
	\end{align*}
	This implies the following.
	\begin{corollary}\label{corollary:equivConstructionAreaSm}
		For $0\le k\le n-1$, the map
		\begin{align*}
			\mathcal{J}^{n}(\R^n)^{tr}&\rightarrow \Area(\R^n)\\
			\tau&\mapsto \Phi_{i_T\tau}
		\end{align*}
		is well-defined, injective, and $\GL(n,\R)$-equivariant in the following sense: For $\tau\in\mathcal{J}^n(\R^n)^{tr}$ and $g\in\GL(n,\R)$, 
		\begin{align*}
			\Phi_{i_T(G_g^{-1})^*\tau}=\sign(\det g) \pi(g)\Phi_{i_T\tau}.
		\end{align*}.
	\end{corollary}
	\begin{proof}
		Since $\tau$ is translation invariant, so is $i_T\tau$, and consequently, $\Phi_{i_T\tau}$ defines an element of $\Area(\R^n)$ by the previous discussion.\\
		If $\tau\in\mathcal{J}^n(\R^n)^{tr}$ lies in the kernel of this map, then $i_T\tau$ belongs to the ideal generated by $\alpha$ and $d\alpha$ by \autoref{lemma:KernelCNC}, and so $\tau=0$ by \autoref{lemma:Kernel_J_to_Area}. \\ It remains to see that it commutes with the natural operation of $\GL(n,\R)$ on both spaces, i.e. that \begin{align*}
			\Phi_{i_T\tau}(g^{-1}K;g^{-1}\cdot\phi)=\sign(\det g)\Phi_{i_T(G_{g^{-1}}^*\tau)}(K;\phi)
		\end{align*}
		for all $K\in\mathcal{K}(\R^n)$ and $\phi\in C(S^{n-1})$. For $\phi\in C^1(S^{n-1})$, this follows from the previous calculation. Since $\Phi_{i_T\tau}(K)$ is a measure, it thus holds for arbitrary continuous functions by approximation.\\
		
	\end{proof}

\section{Continuous area measures}\label{section:AreaGeneral}
	\subsection{Relation to valuations on convex bodies}
	\label{section:RelationValuationProperty}
		\begin{theorem}
			\label{theorem:valuationProperty}
			Let $\Psi:\K(\R^n)\rightarrow\M(S^{n-1})$ be continuous and locally determined. Then $\Psi$ is a valuation.
		\end{theorem}
		\begin{proof}
			Let $K,L\in\K(\R^n)$ be two convex bodies such that $K\cup L$ is convex and consider the sets
			\begin{align*}
				&U=\mathrm{int}\{v\in S^{n-1}: h_K(v)\ge h_L(v)\}, &&W=\mathrm{int}\{v\in S^{n-1}: h_K(v)\le h_L(v)\}.
			\end{align*}
			Then $h_{K\cup L}|_U=h_K|_U$ and $h_{K\cap L}|_U=h_L|_U$, compare \autoref{lemma:propertiesSupportFunction}. If $\phi\in C(S^{n-1})$ satisfies $\supp\phi\subset U$, this implies
			\begin{align*}
				&\Psi(K;\phi)=\Psi(K\cup L;\phi), &&
				\Psi(L;\phi)=\Psi(K\cap L;\phi),
			\end{align*}
			and thus
			\begin{align}
				\label{eq:ValuationProperty}
				\Psi(K;\phi)+\Psi(L;\phi)=\Psi(K\cup L;\phi)+\Psi(K\cap L;\phi).
			\end{align}
			The same reasoning applies if $\supp\phi\subset W$. In particular, if the sets $U,W$ cover $S^{n-1}$, then we may use a partition of unity to see that Eq.~\eqref{eq:ValuationProperty} holds for all $\phi\in C(S^{n-1})$.\\
			If this is not the case, then we define for $\epsilon>0$ convex bodies $K_\epsilon,L_\epsilon\in \K(\R^n)$ by
			\begin{align*}
				&h_{K_\epsilon}=\max(h_K,h_K\wedge h_L+\epsilon), &&h_{L_\epsilon}=\max(h_L,h_K\wedge h_L+\epsilon),
			\end{align*}
			i.e. $K_\epsilon=\mathrm{conv}(K,K\cap L+B_\epsilon(0))$, $L_\epsilon=\mathrm{conv}(L,K\cap L+B_\epsilon(0))$, compare \autoref{lemma:propertiesSupportFunction}. Then  $K_\epsilon\cap L_\epsilon =K\cap L+B_{\epsilon}(0)$ and
			\begin{align*}
				K_\epsilon\cup L_\epsilon =\mathrm{conv}(K\cup L,K\cap L+B_{\epsilon}(0)).
			\end{align*}
			In particular, $K_\epsilon$, $L_\epsilon$, $K_\epsilon\cap L_\epsilon$, and $\K_\epsilon\cup L_\epsilon$ are convex and converge to $K,L,K\cap L$, and $K\cup L$ respectively. By construction, the open sets
			\begin{align*}
				&U_\epsilon=\mathrm{int}\{v\in S^{n-1}: h_{K_\epsilon}(v)\ge h_{L_\epsilon}(v)\}, &&W_\epsilon=\mathrm{int}\{v\in S^{n-1}: h_{K_\epsilon}(v)\le h_{L_\epsilon}(v)\}
			\end{align*}
			cover $S^{n-1}$, so the previous discussion and the continuity of $\Psi$ imply for $\phi\in C(S^{n-1})$,
			\begin{align*}
				\Psi(K;\phi)+\Psi(L;\phi)=&\lim\limits_{\epsilon\rightarrow0}\Psi(K_\epsilon;\phi)+\Psi(L_\epsilon;\phi)\\
				=&\lim\limits_{\epsilon\rightarrow0}\Psi(K_\epsilon\cup L_\epsilon;\phi)+\Psi(K_\epsilon\cap L_\epsilon;\phi)\\
				=&\Psi(K\cup L;\phi)+\Psi(K\cap L;\phi),
			\end{align*}
			which shows the claim.
		\end{proof}

		\begin{remark}
			In general, a locally determined functional $\Psi:\K(\R^n)\rightarrow \M(S^{n-1})$ does not need to to be a valuation. Consider the case $n=2$ and $\Psi$ given for $v_0\in S^1$ fixed by
			\begin{align*}
				\Psi(K)=\begin{cases}
					\delta_{v_0},& h_K~\text{differentiable in}~v_0\in S^1,\\
					0, & \text{else}.
				\end{cases}
			\end{align*}
			This defines a local functional, however, for $u\in v^\perp$, $u\ne 0$, the intervals $[0,1]u$, $[-1,0]u$ satisfy
			\begin{align*}
				\Psi([-1,0]u)+\Psi([0,1]u)=0\ne \delta_{v_0}+0=\Psi(0)+\Psi([-1,1]u).
			\end{align*}
			This is essentially the same example as \cite{KnoerrPolynomiallocalfunctionals2025}*{Example~4.1}.
		\end{remark}

	\subsection{Homogeneous decomposition and Goodey--Weil distributions}\label{section:HomDecompArea}
	By \autoref{theorem:valuationProperty}, $\Area(\R^n)$ may be considered as a subspace of $\Val(\R^n,\M(S^{n-1}))$. In this section, we use this relation to prove \autoref{maintheorem:DirectClassificationResults}. The first part is contained in the following result.
	\begin{theorem}
		\label{theorem:homDecompArea}$\Area(\R^n)=\bigoplus_{k=0}^{n-1}\Area_k(\R^n)$
	\end{theorem}
	\begin{proof}
		By \autoref{theorem:valuationProperty}, $\Area(\R^n)\subset \Val(\R^n,\M(S^{n-1}))$, so for every $\Psi\in \Area(\R^n)$ there exist $\Psi_j\in \Val_j(\R^n,\M(S^{n-1}))$ such that $\Psi=\sum_{j=0}^n\Psi_j$ by \autoref{theorem:McMullenHomDecomp}. In particular, for $t\ge 0$ and $K\in\K(\R^n)$
		\begin{align*}
			\Psi(tK)=\sum_{j=0}^nt^j\Psi_j(K).
		\end{align*}
		We may plug in $t=0,\dots,n$ and use the inverse of the Vandermonde matrix to obtain constants $c_{ij}\in \R$ independent of $K\in\K(\R^n)$ such that
		\begin{align*}
			\Psi_i(K)=\sum_{j=0}^n c_{ij}\Psi(j K)
		\end{align*}
		for every $K\in\K(\R^n)$. Since $K\mapsto \Psi(j K)$ is locally determined, this shows that $\Psi_i$ is locally determined, i.e. $\Psi_i\in \Area_i(\R^n)$. We thus obtain the decomposition 
		\begin{align*}
			\Area(\R^n)=\bigoplus_{k=0}^{n}\Area_k(\R^n).
		\end{align*}
		It remains to see that $\Area_n(\R^n)=0$. For this, we use Hadwiger's volume characterization in \autoref{theorem:HadwigerVolume} to see that any $\Psi\in \Area_n(\R^n)\subset \Val_n(\R^n,\M(S^{n-1}))$ is of the form 
		\begin{align*}
			\Psi(K)=\vol(K)\nu
		\end{align*}
		for some $\nu\in \M(S^{n-1})$, where $\vol$ denotes the Lebesgue measure on $\R^n$. It is easy to see that this is not locally determined if $\nu\ne 0$. Thus $\Psi=0$, which completes the proof.
	\end{proof}
	Note that elements in $\Area_0(\R^n)$ are constant, so $\Area_0(\R^n)\cong \M(S^{n-1})$ via the map $\Psi\mapsto\Psi(0)$. Thus \autoref{maintheorem:DirectClassificationResults} is a consequence of \autoref{theorem:homDecompArea} and the following result.
	
	\begin{theorem}\label{theorem:characterizationTopDegree}
		For every $\Psi\in\Area_{n-1}(\R^n)$ there exists a unique $\phi\in C(S^{n-1})$ such that
		\begin{align*}
			\Psi=\phi\bullet S
		\end{align*}
	\end{theorem}
	\begin{proof}
		If $v_0\in S^{n-1}$ and $P\subset v_0^\perp$ is a polytope, then any $v\in S^{n-1}\setminus\{\pm v_0\}$ has a neighborhood $U_v$ such that $h_P=h_{P'}$ on this neighborhood for an at most $(n-2)$-dimensional face $P'$ of $P$, compare \cite{KnoerrIntegralrepresentationpolynomial2026}*{Lemma~3.5}. In particular,
		\begin{align*}
			\Psi(P)|_{U_v}=\Psi(P')|_{U_v}=0
		\end{align*}
		due to the homogeneous decomposition in \autoref{theorem:McMullenHomDecomp}, so the measure $\Psi(P)$ is concentrated on $\{ v_0,-v_0\}$. By continuity, this holds for any $(n-1)$-dimensional convex body. Hadwiger's characterization in \autoref{theorem:HadwigerVolume} thus implies for $K\subset v_0^\perp$
		\begin{align*}
			\Psi(K)=[\phi(v_0)\delta_v+\phi(-v_0)\delta_{-v}]\vol_{n-1}(K)
		\end{align*}
		for a function $\phi:S^{n-1}\rightarrow \C$. For $v\in S^{n-1}$, let $B_v\subset v^\perp$ denote the $(n-1)$-dimensional unit ball. Then $\Psi(B_v)=[\phi(v)\delta_v+\phi(-v)\delta_{-v}]\omega_{n-1}$, and by integrating $\psi\in C(S^{n-1})$ with $\psi=1$ on a neighborhood of $v$ and $\psi=0$ on a neighborhood of $-v$, it is easy to see that $\phi$ is continuous (as $\Psi$ is continuous in the weak* topology). In particular, $\tilde{\Psi}:=\Psi-\phi\bullet S$ vanishes on all lower dimensional polytopes. If $P\in\K(\R^n)$ is a full dimensional polytope and $v\in S^{n-1}$, then there exists a neighborhood $W_v$ of $v$ such that $h_P=h_{P'}$ on $W_v$ for an $(n-1)$-dimensional face of $P$, compare \cite{KnoerrIntegralrepresentationpolynomial2026}*{Lemma~3.5}. Since $\Psi$ is locally determined, this shows 
		\begin{align*}
			\tilde{\Psi}(P)|_{W_v}=\tilde{\Psi}(P')|_{W_v}=0.
		\end{align*}
		As the sets $W_v$, $v\in S^{n-1}$, cover $S^{n-1}$, this shows $\tilde{\Psi}(P)=0$ for every polytope, and therefore $\tilde{\Psi}$ vanishes identically on $\K(\R^n)$ by continuity, which shows $\Psi=\phi\bullet S$.
	\end{proof}

	Since $\Area_k(\R^n)\subset \Val_k(\R^n,\M(S^{n-1}))$, we may consider the Goodey--Weil distribution of $\Psi$ from \autoref{section:Valuations}. However, it will be more natural to consider a variation of this construction. First, it follows from the Banach-Steinhaus Theorem that the following defines a semi-norm on $\Area(\R^n)$ (compare also \cite{KnoerrPolynomiallocalfunctionals2025}*{Lemma 3.2}):
	\begin{align*}
		\|\Psi\|=\sup_{K\subset B,\|\phi\|_\infty\le 1}|\Psi(K;\phi)|.
	\end{align*}
	This allows for the following localized version of the Goodey--Weil distributions.
	\begin{corollary}
		\label{corollary:modifiedGWArea}
		For every $\Psi\in\Area_k(\R^n)$ there exists a unique distribution $\widehat{\GW}(\Psi):C^\infty((S^{n-1})^{k+1})\rightarrow \C$ such that
		\begin{align}
			\label{eq:defPropertyGWtilde}
			\widehat{\GW}(\Psi)[h_{K_1}\otimes\dots\otimes h_{K_k}\otimes\phi]=[\bar{\Psi}(K_1,\dots,K_k)](\phi).
		\end{align}
		for all smooth convex bodies $K_1,\dots,K_k\in\K(\R^n)$ with strictly positive Gauss curvature, and $\phi\in C^\infty(S^{n-1})$. Moreover, there is a constant $C>0$ such that
		\begin{align*}
			|\widehat{\GW}(\mu)[\phi_1\otimes\dots\otimes \phi_{k+1}]|\le C\|\Psi\|\prod_{j=1}^k \|\phi_j\|_{C^2(S^{n-1})}\cdot \|\phi_{k+1}\|_\infty.
		\end{align*}
	\end{corollary}
	\begin{proof}
		For $\phi_{k+1}\in C(S^{n-1})$ and $\Psi\in \Area_k(\R^n)$, consider $\Psi[\phi_{k+1}]\in\Val(\R^n,\C)$ given by $\Psi[\phi_{k+1}](K)=\Psi(K;\phi_{k+1})$. Then 
		\begin{align*}
			|\GW(\Psi[\phi_{k+1}])[\phi_1\otimes\dots\otimes\phi_k]|\le C\|\Psi[\phi_{k+1}]\|\cdot \prod_{j=1}^k\|\phi_j\|_{C^2(S^{n-1})}
		\end{align*}
		for a constant $C>0$ independent of $\Psi$ and $\phi_j$, $1\le j\le k+1$, by \autoref{corollary:estimateGW}. Since
		\begin{align*}
			\|\Psi[\phi_{k+1}]\|=\sup_{K\subset B_1(0)}	|\Psi[\phi_{k+1}](K)|\le \|\Psi\|\cdot\|\phi_{k+1}\|_\infty,
		\end{align*}
		we obtain
		\begin{align*}
			|\GW(\Psi[\phi_{k+1}])[\phi_1\otimes\dots\otimes\phi_k]|\le C\|\Psi\|\prod_{j=1}^k \|\phi_j\|_{C^2(S^{n-1})}\cdot \|\phi_{k+1}\|_\infty.
		\end{align*} 
		In particular, the Schwartz Kernel Theorem implies that the map 
		\begin{align*}
			(\phi_1,\dots,\phi_{k+1})\mapsto \GW(\Psi[\phi_{k+1}])[\phi_1\otimes\dots\otimes\phi_k]
		\end{align*}
		extends to the desired distribution $\widehat{\GW}(\Psi)$. Due to the defining property of $\GW(\Psi)$ in \autoref{theorem:ExistenceGW}, it is now easy to see that $\widehat{\GW}(\Psi)$ satisfies Eq.~\eqref{eq:defPropertyGWtilde}.
	\end{proof}
	\begin{remark}
	 For $\Psi\in\Area_0(\R^n)\cong \M(S^{n-1})$, we set $\widehat{\GW}(\Psi):=\Psi(0)$.
	\end{remark}
	Similar to the Goodey--Weil distributions, the valuation property combined with the locality properties of $\Psi\in\Area_k(\R^n)$ implies certain support restrictions. Recall that $\Delta_k:\R^n\rightarrow (\R^n)^k$ denotes the diagonal embedding.
	\begin{proposition}
		\label{proposition:SupportGWArea}
		Let $1\le k\le n-1$ and $\Psi\in\Area_k(\R^n)$. Then $\supp\widehat{\GW}(\Psi)=\Delta^{k+1}(\vsupp\Psi)$. 
	\end{proposition}
	\begin{proof}
		If $(v_1,\dots,v_{k+1})\in (S^{n-1})^{k+1}$ belongs to $\supp\widehat{\GW}(\Psi)$, then for any neighborhood $U_j$ of $v_j$ in $S^{n-1}$ there exists $\phi_j\in C^\infty(S^{n-1})$ with $\supp\phi_j\subset U_j$ such that $\widehat{\GW}(\Psi)[\phi_1\otimes\dots\phi_{k+1}]\ne 0$. This implies in particular that $\GW(\Psi)[\phi_1\otimes\dots\otimes\phi_k]\ne0$ and since this holds for every neighborhood of $v_j$, $1\le j\le k$, we obtain that $(v_1,\dots,v_{k})\in \supp\GW(\Psi)=\Delta^{k}(\vsupp\Psi)$. Thus $v_1=\dots=v_k=:v$ due to \autoref{proposition:SupportValuations}. If $v_{k+1}\ne v$, then we find disjoint neighborhoods $U$ of $v_{k+1}$ and $W$ of $v$. If $\phi\in C^\infty(S^{n-1})$ is supported on $U$ and $\phi_1,\dots,\phi_k\in C^\infty(S^{n-1})$ are supported on $W$, then the functions $f_\lambda=h_{B_1(0)}+\sum_{j=1}^k\lambda_j \phi_j$ and $h_{B_1(0)}$ coincide on $U$. For $|\lambda_j|$ small enough, $f_\lambda$ is the support function of a convex body $K_\lambda$. Since $\Psi$ is locally determined, this implies
		\begin{align*}
			\widehat{\GW}(\Psi)[\phi_1\otimes\dots\otimes \phi_{k+1}]
			=&\frac{1}{k!}\frac{\partial^k}{\partial\lambda_1\dots\partial\lambda_k}\Big|_0\Psi\left(K_\lambda;\phi\right)\\
			=&\frac{1}{k!}\frac{\partial^k}{\partial\lambda_1\dots\partial\lambda_k}\Big|_0\Psi\left(B_1(0);\phi\right)=0.
		\end{align*}
		Thus $(v,v_{k+1})\notin \supp\widehat{\GW}(\Psi)$. This shows that $\supp\widehat{\GW}(\Psi)\subset\Delta^{k+1}(\vsupp\Psi)$.\\
		On the other hand, if $v\in \vsupp\Psi$, then for any neighborhood $U\subset S^{n-1}$ of $v$, we have $\phi_j\in C_c(U)$, $1\le j\le k$, such that $\GW(\Psi)[\phi_1,\dots,\phi_k]\ne 0$. With the same argument as above, the measure $\GW(\Psi)[\phi_1,\dots,\phi_k]\in \M(S^{n-1})$ is supported on $U$. As it does not vanish identically, we find $\phi_{k+1}\in C^\infty(S^{n-1})$ with $\supp\phi_{k+1}\subset U$ such that $\widehat{\GW}(\Psi)[\phi_1\otimes\dots\otimes \phi_{k+1}]\ne 0$. Since this holds for any neighborhood $U$ of $v$, $(v,\dots,v)\in \supp \widehat{\GW}(\Psi)$.
	\end{proof}
	
	If $\Psi=\sum_{k=0}^{n-1}\Psi_k$ is the decomposition of $\Psi\in\Area(\R^n)$ into its homogeneous components, then we define the \emph{local vertical support} to be the set 
	\begin{align*}
		\vsupploc\Psi=\bigcup_{k=0}^{n-1}\Delta_{k+1}^{-1}(\supp\widehat{\GW}(\Psi_k)).
	\end{align*}
	\begin{proposition}\label{proposition:characterization_localVsupp}
		For $\Psi\in \Area(\R^n)$, the local vertical support is the unique minimum among all closed subsets $A\subset S^{n-1}$ that satisfy the following properties:
		\begin{enumerate}
			\item $\supp \Psi(K)\subset A$ for all $K\in\K(\R^n)$,
			\item If $h_K=h_L$ on an open neighborhood of $A$, then $\Psi(K)=\Psi(L)$.
		\end{enumerate}
	\end{proposition}
	\begin{proof}
		This follows with the same argument as \cite{KnoerrPolynomiallocalfunctionals2025}*{Corollary~4.17}
	\end{proof}

	\subsection{Three topologies on $\Area(\R^n)$}
	Let us consider $\Area(\R^n)$ as a subspace of $C(\K(\R^n),\M(S^{n-1}))$, where we consider $\M(S^{n-1})=C(S^{n-1})'$ with the weak* topology. Spaces of this type were considered in \cite{KnoerrPolynomiallocalfunctionals2025}*{Section 3.2}, and we will summarize the relevant results needed for the constructions in this article in the following subsections.\\
	
	First, for any compact subset $C\subset \K(\R^n)$ and $B\subset C(S^{n-1})$ bounded, the following semi-norm is finite on $\Area(\R^n)$ by \cite{KnoerrPolynomiallocalfunctionals2025}*{Lemma~3.2}:
	\begin{align*}
		\|\Psi\|_{(C;B)}=\sup_{K\in C, \phi\in B}|\Psi(K;\phi)|.
	\end{align*}
	If $\mathcal{B}$ is a family of bounded sets covering $C(S^{n-1})$, the corresponding families of semi-norms equip $\Area(\R^n)$ with the structure of a locally convex vector space. We will consider the following three families.
	\begin{definition}
		\begin{enumerate}
			\item If $\mathcal{B}$ is the family of all finite subsets, we call the induced topology the uniform weak* topology.
			\item If $\mathcal{B}$ is the family of all relatively compact subsets, we call the induced topology the uniform compact topology.
			\item If $\mathcal{B}$ is the family of all bounded subsets, we call the induced topology the uniform strong topology.
		\end{enumerate}
	\end{definition}
	We refer to \cite{KnoerrPolynomiallocalfunctionals2025}*{Section 3.2} for a more in depth discussion of the properties of these three topologies.\\
	As in the case of translation invariant valuations, the homogeneous decomposition in \autoref{theorem:homDecompArea} has the following consequence.
	\begin{corollary}\label{corollary:AreaBanachSpace}
		$\Area(\R^n)$ is a Banach space with respect to the uniform strong topology. The topology is induced by the norm
		\begin{align*}
			\|\Psi\|=\sup_{K\subset B,\|\phi\|_\infty\le 1}|\Psi(K;\phi)|.
		\end{align*}
	\end{corollary}
	\begin{proof}
		The homogeneous decomposition implies that this defines a norm on $\Area(\R^n)$ and it is easy to check that the induced topology coincides with the uniform strong topology. The relevant completeness property of this topology can be obtained from \cite{KnoerrPolynomiallocalfunctionals2025}*{Lemma~3.9}.
	\end{proof}
	
	\subsection{The action of the general linear group}
	\label{section:ActionGroup}

	The following is a special case of \cite{KnoerrPolynomiallocalfunctionals2025}*{Lemma~3.12 and Proposition~3.13}.
	\begin{lemma}\label{lemma:continuityPropertiesGroupAction}
		The map
		\begin{align*}
			\GL(n,\R)\times \Area(\R^n)&\rightarrow\Area(\R^n)\\
			(g,\Psi)&\mapsto \pi(g)\Psi
		\end{align*}
		is separately continuous with respect to the uniform weak* topology and jointly continuous with respect to the uniform compact topology.
	\end{lemma}
	
	For $\Psi\in \Area(\R^n)$, the map 
	\begin{align}\label{eq:orbitMap}
		\begin{split}
			\GL(n,\R)&\mapsto \Area(\R^n)\\
			g&\mapsto \pi(g)\Psi
		\end{split}
	\end{align}
	is in general not continuous with respect to the uniform strong topology. If this map is continuous, then we call $\Psi$ \emph{strongly $\GL(n,\R)$-continuous}, and we denote by $\Area^0(\R^n)$ the subspace of strongly $\GL(n,\R)$-continuous area measures. Then $\Area^0(\R^n)$ is a closed subspace of $\Area(\R^n)$ with respect to the uniform strong topology, compare \cite{KnoerrPolynomiallocalfunctionals2025}*{Proposition~3.19}.
		\begin{remark}\label{remark:NotStronglyCont}
		For $0\le k\le n-2$, there exist $k$-homogeneous area measures that are not strongly $\GL(n,\R)$-continuous: For a ${k+1}$-dimensional subspace $E\in \Gr_{k+1}(\R^n)$ let $i_E:E\rightarrow\R^n$ and $\pi_E:\R^n\rightarrow E$ denote the natural inclusion and orthogonal projection. Then
		\begin{align*}
			\Psi(K;\phi):=\int_{S(E)}(i_E^*\phi )dS_k(\pi_E(K))
		\end{align*}
		defines an element of $\Area_k(\R^n)$ that does not belong to $\Area^0(\R^n)$. 	For $k=n-1$, every element is strongly $\GL(n,\R)$-continuous, compare \autoref{corollary:topDegreeStronglyCont} below.
		
	\end{remark}
	It is easy to see that $\Area^0(\R^n)=\bigoplus_{k=0}^{n-1}\Area_k^0(\R^n)$ for the corresponding spaces of $k$-homogeneous strongly $\GL(n,\R)$-continuous area measures. The following is a consequence of \cite{KnoerrPolynomiallocalfunctionals2025}*{Corollary~3.20}
	\begin{proposition}\label{proposition:Area0BanachRep}
		The map
		\begin{align*}
			\GL(n,\R)\times \Area^0(\R^n)&\rightarrow\Area^0(\R^n)\\
			(g,\Psi)&\mapsto \pi(g)\Psi
		\end{align*}
		is continuous with respect to the uniform strong topology. In particular, $\Area^0(\R^n)$ is a continuous Banach representation of $\GL(n,\R)$.
	\end{proposition}	

	\subsection{$\GL(n,\R)$-smooth area measures}
		The following is a special case of \cite{KnoerrPolynomiallocalfunctionals2025}*{Corollary~3.17}
		\begin{proposition}
			For $\Psi\in\Area(\R^n)$, consider the map
			\begin{align}
				\label{eq:mapGroupAction}
				\begin{split}
				\GL(n,\R)&\rightarrow\Area(\R^n)\\
				g&\mapsto\pi(g)\Psi.
			\end{split}
			\end{align}
			The following are equivalent:
			\begin{enumerate}
				\item This map is smooth with respect to the uniform weak* topology.
				\item This map is smooth with respect to the uniform compact topology.
				\item This map is smooth with respect to the uniform strong topology.
			\end{enumerate}
		\end{proposition}
		Consequently, $\Psi$ is $\GL(n,\R)$-smooth if it satisfies any of these equivalent conditions. We denote by $\Area^\infty(\R^n)$ the space of $\GL(n,\R)$-smooth area measures. Note that any $\GL(n,\R)$-smooth area measure is automatically strongly $\GL(n,\R)$-continuous, so $\Area^\infty(\R^n)$ is precisely the set of smooth vectors in the Banach representation $\Area^0(\R^n)$ of $\GL(n,\R)$. \cite{KnoerrPolynomiallocalfunctionals2025}*{Corollary~3.22} implies the following.		
		\begin{proposition}\label{proposition:DensitySmoothVectors}
			Let $W\subset\Area(\R^n)$ be a $\GL(n,\R)$-invariant subspace.
			\begin{enumerate}
				\item If $W$ is closed in the uniform weak* or uniform compact topology, then $W\cap \Area^\infty(\R^n)$ is sequentially dense in $W$ with respect to the uniform weak* or uniform compact topology, respectively.
				\item If $W\subset \Area^0(\R^n)$ and if $W$ is closed in the uniform strong topology, then $W\cap \Area^\infty(\R^n)$ is dense in $W$ with respect to the uniform strong topology.
			\end{enumerate}
		\end{proposition}
		
		Note that \autoref{proposition:DensitySmoothVectors} shows in particular that $\Area^0(\R^n)$ is dense in $\Area(\R^n)$ with respect to both the uniform weak* and uniform compact topology, although it is closed in the uniform strong topology.

		We conclude this section with the following results on the $C(S^{n-1})$-module structure. In the following, we equip $C(S^{n-1})$ with an action of $\GL(n,\R)$ different from the one considered in the introduction: We consider elements in $C(S^{n-1})$ as $0$-homogeneous functions on $(\R^n)^*$ and consequently let $g\in \GL(n,\R)$ act on $\phi\in C(S^{n-1})$ by
		\begin{align*}
			[g\phi](v):=\phi\left(\frac{g^{T}v}{|g^{T}v|}\right),\quad v\in S^{n-1}.
		\end{align*}
		\begin{lemma}\label{lemma:continuityModuleStructure}
			The map
			\begin{align*}
				C(S^{n-1})\times \Area(\R^n)&\mapsto \Area(\R^n)\\
				(\phi,\Psi)&\mapsto \phi\bullet \Psi
			\end{align*}
			is $\GL(n,\R)$-equivariant and separately continuous with respect to the uniform weak* and uniform compact topology. It is jointly continuous with respect to the uniform strong topology.
		\end{lemma}
		\begin{proof}
			The equivariance and separate continuity with respect to the given topologies are a simple consequence of the definition. Moreover, we have
			\begin{align*}
				\|\phi\bullet \Psi\|=\sup_{K\subset B_1(0),\|\psi\|_\infty\le1}|\Psi(K,\phi\cdot\psi)|\le \|\phi\|_\infty \|\Psi\|.
			\end{align*}
			Thus this map is continuous in the uniform strong topology.
		\end{proof}

		\begin{corollary}\label{corollary:moduleProductSmoothVectors}
			The map
			\begin{align*}
				C^\infty(S^{n-1})\times \Area^\infty(\R^n)&\mapsto \Area^\infty(\R^n)\\
				(\phi,\Psi)&\mapsto \phi\bullet \Psi
			\end{align*}
			is well defined. Similarly, we have a well defined and continuous map
			\begin{align*}
					C(S^{n-1})\times \Area^0(\R^n)&\mapsto \Area^0(\R^n)\\
					(\phi,\Psi)&\mapsto \phi\bullet \Psi
			\end{align*}
			with respect to the uniform strong topology.
		\end{corollary}
		\begin{proof}
			This follows directly from the fact that this map is $\GL(n,\R)$-equivariant, bilinear, and continuous with respect to the uniform strong topology, together with the observation that $C^\infty(S^{n-1})$ is the set of smooth vectors of the natural representation of $\GL(n,\R)$ on $C(S^{n-1})$.\\
			Similarly, $C(S^{n-1})$ is a continuous representation of $\GL(n,\R)$, so in the second case, the image is contained in the subspace of strongly $\GL(n,\R)$-continuous area measures.
		\end{proof}
		\begin{corollary}\label{corollary:topDegreeStronglyCont}
			Every element in $\Area_{n-1}(\R^n)$ is strongly $\GL(n,\R)$-continuous.
		\end{corollary}
		\begin{proof}
			Since any element in $\Area_{n-1}(\R^n)$ is given by $\phi\bullet S_{n-1}$ for $\phi\in C(S^{n-1})$ by \autoref{theorem:characterizationTopDegree}, \autoref{corollary:moduleProductSmoothVectors} implies that it is sufficient to show that $S_{n-1}$ is strongly $\GL(n,\R)$-continuous. This follows since $\pi(g)S_{n-1}=|\det(g)|^{-1}S_{n-1}$ for any $g\in\GL(n,\R)$, compare Eq.~\eqref{eq:GroupActionSurfaceArea}.
		\end{proof}

			We also obtain the following stronger version of \autoref{corollary:equivConstructionAreaSm}.
			\begin{corollary}
			\label{corollary:ConstructionSmoothAreaNormalCycle}
			The map 
			\begin{align*}
				\mathcal{J}^n(\R^n)^{tr}&\rightarrow\Area^\infty(\R^n)\\
				\tau&\mapsto \Phi_{i_T\tau}
			\end{align*}
			is well-defined and injective.
		\end{corollary}
		\begin{proof}
			Due to \autoref{corollary:equivConstructionAreaSm}, we only need to check that $\Phi_{i_T\tau}$ is a $\GL(n,\R)$-smooth area measure.	Note that the map
			\begin{align*}
				\mathcal{J}^n(\R^n)^{tr}&\rightarrow\Area(\R^n)\\
				\tau&\mapsto \Phi_{i_T\tau}
			\end{align*}
			is continuous with respect to the uniform strong topology due to \autoref{proposition:boundSmoothArea}. Since this map commutes with the natural operation of $\GL(n,\R)$ up to the locally constant function $g\mapsto \sign(\det g)$, $\mathcal{J}^n(\R^n)^{tr}$ is a smooth Fr\'echet representation of $\GL(n,\R)$, and the map is linear as well as continuous with respect to uniform strong topology on $\Area(\R^n)$, its image consists of $\GL(n,\R)$-smooth vectors in $\Area(\R^n)$. 
		\end{proof}

\section{Characterization of $\GL(n,\R)$-smooth area measures}\label{section:CharcSmoothArea}
	\subsection{Relation to local functionals on convex functions and vertical localization}
	As discussed in the introduction, the characterization of $\GL(n,\R)$-smooth area measures relies on a corresponding regularity result obtained in \cite{KnoerrPolynomiallocalfunctionals2025,KnoerrIntegralrepresentationpolynomial2026} for the following class of local functionals on convex functions:
	\begin{definition}
		We denote by $\LV(\R^n)$ the space of all maps $\Psi:\Conv(\R^n,\R)\rightarrow\M(\R^n)$ with the following properties:
		\begin{enumerate}
			\item $\Psi$ is continuous with respect to the weak* topology on $\M(\R^n)\cong C_c(\R^n)'$.
			\item $\Psi$ is locally determined: If $f,h\in\Conv(\R^n,\R)$ are two functions with $f|_U=h|_U$ on an open subset $U\subset\R^n$, then $\Psi(f)|_U=\Psi(h)|_U$.
			\item $\Psi$ is dually epi-translation invariant: $\Psi(f+\ell)=\Psi(f)$ for every $f\in\Conv(\R^n,\R)$, $\ell:\R^n\rightarrow\R$ affine.
		\end{enumerate}
	\end{definition}
	We equip $\LV(\R^n)$ and its subspaces with the topology induced from the semi-norms
	\begin{align*}
		\Psi\mapsto \sup_{f\in K}|\Psi(f;\phi)|
	\end{align*}
	where $K\subset\Conv(\R^n,\R)$ is a compact subset and $\phi\in C_c(\R^n)$, which we call the \emph{compact-to-weak*} topology. 
	
In \cite{KnoerrPolynomiallocalfunctionals2025}, a notion of support for elements in $\LV(\R^n)$ was introduced similar to the local vertical support considered in \autoref{section:HomDecompArea}. We will omit its precise definition and only state the following characterization (which may also serve as a definition).
	\begin{proposition}[\cite{KnoerrPolynomiallocalfunctionals2025}*{Corollary~4.17.}]
		\label{proposition:chacterizationLocSupp}
		For $\Psi\in \LV(\R^n)$, its local support  $\supploc\Psi$ is the unique minimum among all closed subsets $A\subset\R^n$ with the following properties:
		\begin{enumerate}
			\item $\supp\Psi(f)\subset A$ for all $f\in\Conv(\R^n,\R)$.
			\item If $f,h\in\Conv(\R^n,\R)$ satisfy $f=h$ on an open neighborhood of $A$, then $\Psi(f)=\Psi(h)$.
		\end{enumerate}
	\end{proposition}
	
	Let $S^n_-=\{v=(v_1,\dots,v_{n+1})\in S^{n}: v_{n+1}<0\}$ denote the negative open half sphere. We have a natural map $I:C_c(\R^n)\rightarrow C_c(S^n_-)$ defined by associating to $\phi\in C_c(\R^n)$ the $1$-homogeneous function on $\R^n\times \R$ given by
	\begin{align*}
		(y,t)\mapsto \begin{cases}
			0& t\ge 0,\\
			|t|\phi\left(\frac{y}{|t|}\right)& t<0,
		\end{cases}
	\end{align*}
	and letting $I(\phi)$ denote the restriction to $S^n$. 
	Consider the diffeomorphism
	\begin{align*}
		P:\R^n&\rightarrow S^n_-\\
		y&\mapsto \frac{(y,-1)}{\sqrt{1+|y|^2}}.
	\end{align*}
	Its inverse is given by $(y,t)\mapsto \frac{y}{|t|}$, so we have $[I(\phi)](y,t)= |t| [\phi\circ P^{-1}](y,t)$ for $\phi\in C_c(\R^n)$. In particular $[I^{-1}(\psi)]=\frac{1}{\sqrt{1+|y|^2}}[\psi\circ P](y)$ for $\psi\in C_c(S^n_-)$.

	Consider the homomorphism $H:\R^n\rightarrow \GL(n+1,\R)$ given by
	\begin{align*}
		x\mapsto \begin{pmatrix}
			I_n & 0\\
			-x^T & 1
		\end{pmatrix}.
	\end{align*}
	We will consider the group of translations of $\R^n$ as a subgroup of $\GL(n+1,\R)$ using this identification. Tracing through the definitions, we obtain the following.
	
	\begin{lemma}\label{lemma:equivariance1homExtension}
		The map $I:C_c(\R^n)\rightarrow C(S^n)$ is equivariant in the following sense: If $t_x:\R^n\rightarrow\R^n$ denotes the translation $y\mapsto y+x$, then
		\begin{align*}
			[I(\phi\circ t_x)](v,t)=|H(x)^T(v,t)|I(\phi)\left(\frac{H(x)^T(v,t)}{|H(x)^T(v,t)|}\right)\quad\text{for}~(v,t)\in S^{n},
		\end{align*}
	\end{lemma}

	\begin{corollary}\label{corollary:adjoint1HomFunction}
		Let $A\subset \R^n$ be compact. The adjoint $I^*:\M(S^n)\rightarrow \M(\R^n)$ restricts to a well defined and bijective map 
		\begin{align*}
			\{T\in \M(S^n):\supp T\subset P(A)\}\rightarrow \{\nu\in \M(\R^n):\supp \nu\subset A\}
		\end{align*}
		which is a topological isomorphism if both spaces are equipped with the weak* topology.
	\end{corollary}	
	\begin{proof}
		This follows since $I^*$ and $(I^*)^{-1}$ are given by pushforwards along diffeomorphisms up to multiplication with a continuous function. The only thing to note is that the topology of $\{T\in \M(S^n):\supp T\subset P(A)\}$, where $\M(S^n)\cong C(S^n)'$ is equipped with the weak* topology, is the same as the topology of $\{T\in \M(S_-^n):\supp T\subset P(A)\}$ for $\M(S_-^n)=C_c(S^n_-)'$ equipped with the weak* topology. This is the case since we can choose $\psi\in C_c(S^n_-)$ with $\psi=1$ on the compact set $P(A)$ and multiply any function in $C(S^n)$ with $\psi$ without affecting the values of the measures evaluated in these functions.  
	\end{proof}
	
	Let us denote by $\LV_c(\R^n)$ the subspace of all $\Psi\in \LV(\R^n)$ such that $\supploc\Psi$ is compact. If $A$ is a compact set with $\supploc\Psi\subset A$ , then the support of $\Psi(h_{K}(\cdot,-1))$ is also contained in $A$ independent of $K\in\K(\R^n\times\R)$, compare \autoref{proposition:characterization_localVsupp}. In particular, \autoref{corollary:adjoint1HomFunction} shows that the map
	\begin{align*}
		T(\Psi):\K(\R^n\times\R)&\rightarrow \M(S^n)\\
		K&\mapsto (I^{-1})^*(\Psi(h_{K}(\cdot,-1)))
	\end{align*}
	is well defined and continuous. From the definition, we directly obtain that $T(\Psi)$ is locally determined. Moreover, \autoref{lemma:propertiesSupportFunction} shows that $T(\Psi)$ is translation invariant, so $T(\Psi)\in \Area(\R^n\times\R)$.

	\begin{theorem}\label{theorem:isomorphismToLocalFunctionals}
		The map $T:\LV_{c}(\R^n)\rightarrow\{\Psi\in\Area(\R^n\times\R):\vsupploc\Psi\subset S^{n}_-\}$ is bijective. For every compact subset $A\subset \R^n$, its restriction
		\begin{align*}
			T:&\{\Psi\in\LV(\R^n):\supploc\Psi\subset A\}\\
			&\rightarrow\{\Psi\in\Area(\R^n\times\R):\vsupploc\Psi\subset P(A)\}
		\end{align*}
		is a topological isomorphism with respect to the uniform weak* topology.
	\end{theorem}
	\begin{proof}
		Using the characterization of the supports in \autoref{proposition:characterization_localVsupp} and \autoref{proposition:chacterizationLocSupp}, it is easy to check that this map is well defined. Moreover, since functions of the form $h_K(\cdot,-1)$ form a dense subset of $\Conv(\R^n,\R)$ (compare, for example, \cite{Knoerrsupportduallyepi2021}*{Corollary~2.10}), this map is injective. Using \autoref{corollary:adjoint1HomFunction}, we may interpret the set on the right hand side as a subset of $\Val(\R^n\times\R,\M_{A}(\R^n))$, where $\M_{A}(\R^n)$ denotes the space of complex measures on $\R^n$ with support in $A$. Under this identification, we obtain functionals with vertical support contained in $P(A)$, so as in  \cite{Knoerrsupportduallyepi2021}*{Proposition~6.17}, we may extend any $\Phi\in\Area(\R^n\times\R)$ with $\vsupploc\Phi\subset P(A)$ to $\Conv(\R^n,\R)$ in the following way: For every $f\in\Conv(\R^n,\R)$, fix a convex body $K\in\K(\R^n\times\R)$ with $f=h_K(\cdot,-1)$ on a neighborhood of $A$ (such a body exists, compare \cite{Knoerrsupportduallyepi2021}*{Proposition 2.9}), and set 
		\begin{align*}
			\Psi(f):=(I^{-1})^*[\Phi(K_f)].
		\end{align*}
		Then $\Psi:\Conv(\R^n,\R)\rightarrow \M(\R^n)$ is well defined and continuous with respect to the weak* topology, compare \cite{Knoerrsupportduallyepi2021}*{Proposition~6.17}. Since this extension does not depend on the precise choice of bodies $K_f$, it is now easy to check that $\Psi$ is locally determined and dually epi-translation invariant. Moreover, the characterization in \autoref{proposition:chacterizationLocSupp} directly shows that $\supploc\Psi\subset A$. Thus the map $T$ is surjective.\\
		In order to see that this map is a topological isomorphism, note first that we may choose $\psi\in C_c(S^n_-)$ with $\psi= 1$ on a neighborhood of $P(A)$. Then for any $\Psi\in \LV(\R^n)$ with $\supp\Psi\subset A$ and $\phi\in C(S^{n-1})$
		\begin{align*}
			&\sup_{K\subset B_1(0)}|[T(\Psi)](K;\phi)|=\sup_{K\subset B_1(0)}|[T(\Psi)](K;\psi\phi)|\\
			=&\sup_{K\subset B_1(0)}|\Psi(h_K(\cdot, -1);I^{-1}(\psi\phi)|.
		\end{align*}
		Since the set $\{h_K(\cdot,-1):K\subset B_1(0)\}$ is compact in $\Conv(\R^n,\R)$, this shows that the restriction of $T$ is continuous. For the converse, let $C\subset\Conv(\R^n,\R)$ be a compact subset and choose $R>0$ such that $A\subset B_R(0)$. It then follows from \cite{Knoerrsupportduallyepi2021}*{Proposition~2.4} that
		\begin{align*}
			c:=\sup_{f\in C, x\in B_{R+1}(0)}|f(x)|<\infty.
		\end{align*}
		Set
		\begin{align*}
			K_f:=\epi(f^*)\cap\{(y,t)\in\R^n\times \R:|y|\le 2c,~|t|\le(2R+3)c	\},
		\end{align*}
		where $\epi(f^*)$ denotes the epi-graph of the Legendre transform of $f$. Then $K_f$ is a convex body and $f=h_{K_f}(\cdot,-1)$ on $B_{R+1}(0)$ by \cite{Knoerrsupportduallyepi2021}*{Proposition~2.9}. Moreover, the set $\{K_f:f\in C\}$ is compact in $\K(\R^{n+1})$ by the Blaschke Selection Theorem,  since it is closed and bounded. By construction, for $\phi\in C_c(\R^n)$ and $\Phi\in \Area(\R^n\times\R)$ with $\vsupploc\Phi\subset P(A)$,
		\begin{align*}
			[T^{-1}(\Phi)](f;\phi)=\Phi(K_f;I(\phi))
		\end{align*}
		for all $f\in C$. In particular,
		\begin{align*}
			\sup_{f\in C}|[T^{-1}(\Phi)](f;\phi)|=\sup_{K\in\{K_f:f\in C\}}|\Phi(K;I(\phi))|
		\end{align*}
		independent of $\Phi$, so $T^{-1}$ is continuous.
	\end{proof}
	We have a natural action of $\R^n$ on $\LV(\R^n)$ given as follows: For $\Psi\in\LV(\R^n)$ and $x\in\R^n$, let $\tilde{\pi}(x)\Psi\in\LV(\R^n)$ be defined by
	\begin{align*}
		[\tilde{\pi}(x)\Psi](f;\phi)=\Psi(f\circ t_x;\phi\circ t_x)
	\end{align*}
	for $f\in\Conv(\R^n,\R)$, $\phi\in C_c(\R^n)$, where $t_x:\R^n\rightarrow\R^n$ denotes the translation by $x\in\R^n$. Note that this action preserves the subspace of compactly supported local functionals. From the definition of $T$ and \autoref{lemma:equivariance1homExtension}, we directly obtain the following relation.
	\begin{lemma}\label{lemma:equivarianceT}
		For $\Psi\in \LV_c(\R^n)$ and $x\in\R^n$, we have
		\begin{align*}
			T(\tilde{\pi}(x)\Psi)=\pi(H(x))T(\Psi).
		\end{align*}
	\end{lemma}
	 As in \cite{KnoerrPolynomiallocalfunctionals2025,KnoerrIntegralrepresentationpolynomial2026}, let us call $\Psi\in \LV(\R^n)$ smooth if the map
	\begin{align*}
		\R^n&\rightarrow\LV(\R^n)\\
		x&\mapsto \tilde{\pi}(x)\Psi
	\end{align*}
	is smooth with respect to the compact-to-weak* topology. The previous result has the following consequence.
	\begin{corollary}\label{corollary:EquivalenceSmoothnessUnderT}
		Let  $\Psi\in\LV_{c}(\R ^n)$. If $T(\Psi)\in \Area(\R ^n\times\R)$ is a $\GL(n+1,\R)$-smooth area measure, then $\Psi$ is a smooth local functional.
	\end{corollary}
	\begin{proof}
		Due to \autoref{lemma:equivarianceT}, we have for $x\in\R^n$,	
		\begin{align*}
			\tilde{\pi}(x)\Psi=T^{-1}\left(\pi(H(x))T(\Psi)\right).
		\end{align*}
		It is sufficient to show that the right hand side is smooth on a neighborhood of $0$. If we fix a compact neighborhood $U$ of $0$, then $\tilde{\pi}(t_x)\Psi$ is supported on a compact set $A$ for every $x\in U$, so the area measure $\pi(H(t_x))T(\Psi)$ is supported on $P(A)\subset S^n_-$ for every $x\in U$. In particular, \autoref{theorem:isomorphismToLocalFunctionals} shows that the left hand side of this equation is smooth on $U$ if and only if $x\mapsto \pi(H(t_x))T(\Psi)$ is smooth on $U$, since $T^{-1}$ is a topological isomorphism between the relevant subspaces of functionals with support in $P(A)$ and $A$, respectively. If $T(\Psi)$ is a smooth area measure, this implies that $x\mapsto \tilde{\pi}(t_x)\Psi$ is smooth on a neighborhood of $0$. Thus $\Psi\in\LV_c(\R^n)$ is smooth. 
	\end{proof}
	
	\subsection{Representation by integration with respect to the normal cycle and proof of \autoref{maintheorem:IntegralRepresentation}}
	In \cite{KnoerrMongeAmpereoperators2024,KnoerrPolynomiallocalfunctionals2025,KnoerrIntegralrepresentationpolynomial2026}, the following construction of elements in $\LV(\R^n)$ was considered. Let us denote by $D(f)$ the differential cycle of $f\in\Conv(\R^n,\R)$ as introduced by Fu \cite{FuMongeAmperefunctions.1989}. This is an integral $n$-current on the cotangent bundle $T^*\R^n=\R^n\times(\R^n)^*$, which is given by integration over the graph of $df$ if $f$ is $C^2$. For any smooth differential form $\tau\in \Omega^n(T^*\R^n)$, we may associate to $f\in\Conv(\R^n,\R)$ the measure on $\R^n$ given by
	\begin{align*}
		\Psi_\tau(f;B) =D(f)[1_{\pi^{-1}(B)}\tau],
	\end{align*}
	where $\pi:\R^n\times(\R^n)^*\rightarrow\R^n$ denotes the projection onto the first factor. As shown in \cite{KnoerrMongeAmpereoperators2024}*{Theorem 4.10}, the map $\Psi_\tau:\Conv(\R^n,\R)\rightarrow\M(\R^n)$ is continuous with respect to the weak* topology and locally determined. Moreover, if $\tau$ is invariant with respect to translations in the second component of $\R^n\times(\R^n)^*$, then $\Psi_\tau$ is dually epi-translation invariant and therefore belongs to $\LV(\R^n)$. A characterization of this subspace was obtained in \cite{KnoerrIntegralrepresentationpolynomial2026}. Under the support restrictions relevant to our purposes, this result reads as follows.	
	\begin{theorem}[\cite{KnoerrIntegralrepresentationpolynomial2026}*{Proposition~3.24}]
		\label{theorem:equivalenceSmoothnessLV}
		The following are equivalent for $\Psi\in \LV_{c}(\R^n)$ and a compact and convex set $A\subset \R^n$:
		\begin{enumerate}
			\item $\Psi$ is smooth and $\supploc\Psi\subset A$
			\item For any basis $\tau_j$ of $\Lambda^n(\R^n\times(\R^{n})^*)^*$ there exist $\phi_j\in C_c^\infty(\R^n)$ with $\supp\phi_j\subset A$ such that
			\begin{align*}
				\Psi(f;B)=D(f)\left[1_{\pi^{-1}(B)}\sum_{j}\pi^*\phi_j\wedge \tau_j\right]
			\end{align*}
			for $f\in\Conv(\R^n,\R)$, $B\subset \R^n$ bounded Borel set.
		\end{enumerate}
	\end{theorem}
	\begin{remark}
		For constant differential forms $\tau\in\Lambda^n(\R^n\times(\R^n)^*)^*$, the elements $\Psi_\tau$ can be identified with elements belonging to a certain space $\MAVal(\R^n)$ of translation equivariant Monge--Amp\`ere operators, see \cite{KnoerrMongeAmpereoperators2024}. \cite{KnoerrIntegralrepresentationpolynomial2026}*{Proposition~3.24} uses this identification in the statement of the result.
	\end{remark}
	
	Consider the map
	\begin{align*}
		Q:(\R^n\times \R)\times S^n_-&\rightarrow \R^n\times\R^n\cong T^*\R^n\\
		\left((y,s),(x,t)\right)&\mapsto \left(\frac{x}{|t|},y\right).
	\end{align*}		
	\begin{proposition}[\cite{KnoerrSmoothvaluationsconvex2024}*{Proposition 6.1}]
		\label{proposition:relation_differentialCycleNormalCycle}
		Let $K\in\K(\R^n\times\R)$. Then
		\begin{align*}
			Q_*\left[\nc(K)|_{(\R^n\times \R)\times S^n_-}\right]=(-1)^nD\left(h_K(\cdot,-1)\right)
		\end{align*}
	\end{proposition}
	\begin{remark}
		The formula in \cite{KnoerrSmoothvaluationsconvex2024}*{Proposition 6.1} differs from the one given above by a sign, which is a consequence of a choice of orientation different from the standard orientation on $\R^{n+1}$.
	\end{remark}
	\begin{corollary}\label{corollary:smoothAreaHemisphere}
		If $\Psi\in\Area(\R^n)$ is $\GL(n,\R)$-smooth and such that $\vsupploc\Psi$ is contained in an open hemisphere, then there exists a differential form $\omega\in \Omega^{n-1}(\R^n\times S^{n-1})^{tr}$ such that $\Psi=\Phi_\omega$.
	\end{corollary}
	\begin{proof}
		Without loss of generality, we may assume that $\vsupploc\Psi$ is a compact subset of $S^{n-1}_-$. From \autoref{theorem:isomorphismToLocalFunctionals}, we obtain $\tilde{\Psi}\in \LV_c(\R^{n-1})$ with $T(\tilde{\Psi})=\Psi$. Then $\tilde{\Psi}$ is smooth by \autoref{corollary:EquivalenceSmoothnessUnderT}, and so \autoref{theorem:equivalenceSmoothnessLV} implies that there exists a differential form $\tilde{\omega}\in \Omega^*_c(\R^{n-1})\otimes\Lambda^*((\R^{n-1})^*)^*$ of degree $n-1$ such that $\tilde{\Psi}(f;\phi)=D(f)[\pi^*\phi\wedge \tilde{\omega}]$. In particular, we may extend the differential form $(-1)^nQ^*\tilde{\omega}\in \Omega^{n-1}(\R^n\times S^{n-1}_-)^{tr}$ by $0$ to $\R^n\times S^{n-1}$ to obtain a well defined smooth differential form $\omega\in \Omega^{n-1}(\R^n\times S^{n-1})^{tr}$ with support contained in $\R^n\times S^{n-1}_-$. Tracing through the definitions, \autoref{proposition:relation_differentialCycleNormalCycle} shows that $\Psi=T(\tilde{\Psi})=\Phi_{\omega}$, which completes the proof.
	\end{proof}

	\begin{proof}[Proof of \autoref{maintheorem:IntegralRepresentation}]
		If $\Psi\in \Area^\infty(\R^n)$, choose a finite open cover of $S^{n-1}$ by open hemispheres and fix a partition of unity $\phi_1,\dots,\phi_N\in C^\infty(S^{n-1})$ subordinate to this cover. Then
		\begin{align*}
			\Psi=\sum_{j=1}^N \phi_j\bullet \Psi,
		\end{align*}
		where $\vsupploc(\phi_j\bullet\Psi)$ is contained in an open hemisphere. Moreover, $\phi_j\bullet\Psi$ is $\GL(n,\R)$-smooth as well by \autoref{corollary:moduleProductSmoothVectors}, so \autoref{corollary:smoothAreaHemisphere} shows that there exists a differential form $\omega_j\in \Omega^{n-1}(\R^n\times S^{n-1})^{tr}$ such that $\phi_j\bullet \Psi=\Phi_{\omega_j}$. Since this holds for all $1\le j\le N$, we obtain $\Psi=\Phi_{\sum_{j=1}^N\omega_j}$.
	\end{proof}
	
	Combining \autoref{maintheorem:IntegralRepresentation} and \autoref{corollary:ConstructionSmoothAreaNormalCycle}, we obtain the following.
	\begin{corollary}\label{corollary:PsiIsomorphism}
		The map
		\begin{align*}
			\mathcal{J}^{k+1,n-1-k}(\R^n)^{tr}&\rightarrow \Area_k^\infty(\R^n)\\
			\tau&\mapsto \Phi_{i_T\tau}
		\end{align*}
		is an isomorphism of vector spaces.
	\end{corollary}
	
	In order to obtain \autoref{maincorollary:InvariantAreaMeasures} from this description, we need to show that every continuous $G$-equivariant area measure is $\GL(n,\R)$-smooth. This requires some some additional background for representations of $\GL(,n,\R)$ and its compact subgroups, and we refer to \cite{SepanskiCompactLiegroups2007} for a general overview and to \cite{WarnerHarmonicanalysissemi1972} as a general reference. Recall that a representation of $\GL(n,\R)$ on a locally convex vector spaces $V$ is called continuous if the map
	\begin{align*}
		\GL(n,\R)\times V&\rightarrow V\\
		(g,v)&\mapsto g\cdot v
	\end{align*}
	is jointly continuous. For a compact subgroup $G\subset \GL(n,\R)$, let us denote by $V[\pi]$ the $\pi$-isotypical component of an irreducible representation $(E_\pi,\pi)$ of $G$, i.e. the sum of all subrepresentations of $V$ that are isomorphic to $E_\pi$. Note that since $G$ is compact, every irreducible representation of $G$ is finite dimensional. Let $V^\infty$ denote the subspace of $\GL(n,\R)$-smooth vectors in $V$. If $V$ is quasi-complete, then $V[\pi]\cap V^\infty$ is a dense subspace of $V[\pi]$, compare \cite{WarnerHarmonicanalysissemi1972}*{Theorem~4.4.3.1}. If $\dim V[\pi]$ is finite dimensional for every irreducible representation of $G$, then $V$ is called a $G$-admissible representation. In this case, $V[\pi]\cap V^\infty=V[\pi]$, which implies that every $G$-finite element in $V$ is $\GL(n,\R)$-smooth, where $v\in V$ is called $G$-finite if its orbit under $G$ spans a finite dimensional subspace.
	
	\begin{proposition}\label{proposition:admissibility}
		Let $G\subset \SO(n)$ be a compact subgroup that operates transitively on $S^{n-1}$. Then $\Area(\R^n)$ is a $G$-admissible representation of $\GL(n,\R)$. In particular, any $G$-finite element in $\Area(\R^n)$ is $\GL(n,\R)$-smooth
	\end{proposition}
	\begin{proof}
		Recall that $\Area(\R^n)$ is a continuous representation of $\GL(n,\R)$ with respect to the uniform compact topology by \autoref{lemma:continuityPropertiesGroupAction}. Since $\Area(\R^n)$ is quasi-complete with respect to the uniform compact topology (compare \cite{KnoerrPolynomiallocalfunctionals2025}*{Lemma 3.8}), the subspace
		\begin{align}
			\label{eq:Admissibility}
			(\Area(\R^n))[\pi]\cap \Area^\infty(\R^n)\subset 
			(\Area(\R^n))[\pi]
		\end{align}
		is dense in any isotypical component. Since $G$ operates transitively on the unit sphere, the $\pi$-isotypical component of $\mathcal{J}^n(\R^n)^{tr}\subset \Omega^n(\R^n\times S^{n-1})^{tr}$ is finite dimensional, as $\mathcal{J}^n(\R^n)^{tr}$ may be identified with a representation induced from a finite dimensional representation of the stabilizer of any point $(0,v)\in \R^n\times S^{n-1}$ in $G$, which is $G$-admissible by Frobenius reciprocity (see, e.g., \cite{SepanskiCompactLiegroups2007}*{Theorem~7.47}). \autoref{corollary:PsiIsomorphism} thus implies that $(\Area(\R^n))[\pi]\cap \Area^\infty(\R^n)$ is finite dimensional and, consequently, its closure is finite dimensional as well. Thus we have equality in Eq.~\eqref{eq:Admissibility}, so $\Area(\R^n)$ is in particular $G$-admissible and every $G$-finite element in $\Area(\R^n)$ is $\GL(n,\R)$-smooth.
	\end{proof}
	
	\begin{proof}[Proof of \autoref{maincorollary:InvariantAreaMeasures}]
			This follows directly from \autoref{maintheorem:IntegralRepresentation} and \autoref{proposition:admissibility} since $\Area(\R^n)^G$ is the isotypical component corresponding to the trivial representation of $G$.
	\end{proof}
	
	\begin{corollary}
		Let $G\subset \SO(n)$ be a compact subgroup. Then $\Area(\R^n)^G$ is finite dimensional if and only if $G$ operates transitively on the unit sphere.
	\end{corollary}
	\begin{proof}
		If $G$ operates transitively on the unit sphere, then $\Area(\R^n)^G$ is finite dimensional by \autoref{proposition:admissibility} since this is the isotypical component of $\Area(\R^n)$ corresponding to the trivial representation. On the other hand, if $G$ does not operate transitively on the unit sphere, then $\Area_0(\R^n)^G\cong \M(S^{n-1})^G$ is infinite dimensional.
	\end{proof}

\section{Applications}\label{section:Applications}
	
	\subsection{Invariant $C(S^{n-1})$-submodules}
	We are going to obtain \autoref{maintheorem:IrreducibleModule} from the following result for $\GL(n,\R)$-smooth area measures.
	
	\begin{theorem}\label{theorem:IrreducibilityModuleSmoothCase}
		Let $A\subset \Area^\infty_k(\R^n)$ be a $\GL(n,\R)$-invariant $C^\infty(S^{n-1})$-module. Then $A=\Area^\infty_k(\R^n)$.
	\end{theorem}
	\begin{proof}
		Since $\Area_k^\infty(\R^n)\cong \mathcal{J}^{k+1,n-1-k}(\R^n)^{tr}$ as a $C^\infty(S^{n-1})$-module and as the isomorphism commutes with the action of $\GL(n,\R)$ on both spaces up to the tensor product with the $1$-dimensional representation of $\GL(n,\R)$ given by $g\mapsto \sign(\det g)$, compare \autoref{corollary:equivConstructionAreaSm} and \autoref{corollary:PsiIsomorphism}, the claim follows directly from \autoref{theorem:IrreducibilityDiffForms}.
	\end{proof}		
	The following is a refined version of  \autoref{maintheorem:IrreducibleModule}.	
	\begin{theorem}\label{theorem:IrreducibilityRefined}
		Let $W\subset \Area_k(\R^n)$ be a nontrivial $\GL(n,\R)$-invariant $C(S^{n-1})$-module. Then the following holds:
		\begin{enumerate}
			\item $W$ is dense in $\Area_k(\R^n)$ with respect to the uniform weak* and uniform compact topology.
			\item If $W\subset \Area_k^0(\R^n)$, then $W$ is dense in $\Area_k^0(\R^n)$ with respect to the uniform strong topology.
		\end{enumerate}
	\end{theorem}
	\begin{proof}
		The argument is identical for all three topologies, the only difference is that the closure of $\Area_k^\infty(\R^n)$ with respect to the uniform strong topology is the (in general proper) subspace $\Area_k^0(\R^n)\subset \Area_k(\R^n)$, which leads to the additional restrictions in the second case. Let us fix one of these topologies. Due to \autoref{lemma:continuityPropertiesGroupAction} and \autoref{lemma:continuityModuleStructure}, the closure of $W$ with respect to the chosen topology is again a $\GL(n,\R)$-invariant $C(S^{n-1})$-module. Without loss of generality, we may thus assume that $W$ is closed in the chosen topology. For the uniform strong topology, this closure is still contained in $\Area^0_k(\R^n)$ since $\Area^0_k(\R^n)$ is closed in the uniform strong topology due to \cite{KnoerrPolynomiallocalfunctionals2025}*{Proposition~3.19}.\\
		By \autoref{proposition:DensitySmoothVectors}, $\Area_k^\infty(\R^n)\cap W$ is dense in $W$ with respect to the chosen topology. By construction, $\Area_k^\infty(\R^n)\cap W$ is a nontrivial $\GL(n,\R)$-invariant $C^\infty(S^{n-1})$-module, so \autoref{theorem:IrreducibilityModuleSmoothCase} implies that $\Area_k^\infty(\R^n)\subset W$. Since $\Area_k^\infty(\R^n)$ is dense in $\Area_k(\R^n)$ with respect to the uniform weak* and uniform compact topology as well as dense in $\Area^0_k(\R^n)$ with respect to the uniform strong topology, this implies the claim.
	\end{proof}

	\subsection{Finite linear combinations of mixed area measures and McMullen's Conjecture}\label{section:FiniteComb}
	
	Recall that we denote the $1$-homogeneous extension of a real-valued function $h\in C^\infty(S^{n-1})$, by $\tilde{h}:\R^n\rightarrow \R$. In \autoref{section:ConstructionSmoothArea}, we considered the vector field $X_h$ on $\R^n\times S^{n-1}$ given by
	\begin{align*}
		X_h|_{(x,v)}:=\sum_{j=1}^n \frac{\partial \tilde{h}}{\partial v_i}(v) \frac{\partial }{\partial x_i}.
	\end{align*}
	In other words $X_h|_{(x,v)}=(\nabla \tilde{h}(v),0)\in T_{(x,v)}(\R^n\times S^{n-1})\cong \R^n\oplus v^\perp$. Since $\tilde{h}$ is $1$-homogeneous, this vector field may also be characterized as the unique vector field such that
	\begin{align*}
		&i_{X_h}\alpha=\pi_2^*h, &&i_{X_h}d\alpha=-d\pi_2^*h.
	\end{align*}
	If $\mathcal{L}_{X_h}$ denotes the Lie derivative with respect to $X_h$, this implies $\mathcal{L}_{X_h}\alpha=0$ and $\mathcal{L}_{X_{h}}d\alpha=0$. Moreover, for $\omega\in \Omega^{k,l}(\R^n\times S^{n-1})^{tr}$, $\mathcal{L}_{X_h}\omega\in \Omega^{k-1,l+1}(\R^n\times S^{n-1})^{tr}$. In particular, we obtain a well defined map 
	\begin{align*}
		\mathcal{L}_{X_h}:\mathcal{J}^{k,n-k}(\R^n)^{tr}\rightarrow\mathcal{J}^{k-1,n-k+1}(\R^n)^{tr}.
	\end{align*}
	The flow of this vector field admits the following simple description: Consider the map
	\begin{align*}
		\phi_{h}:\R^n\times S^{n-1} &\rightarrow \R^n\times S^{n-1}\\
		(x,v)&\mapsto (x+\nabla\tilde{h}(v),v).
	\end{align*}
	Then $\phi_h$ is a contactomorphism and the flow of $X_h$ is given by $t\mapsto \phi_{th}$. In particular, $\frac{d}{dt}\phi_{th}|_{t=0}=X_h$. Moreover, since the maps $\phi_{h_1}$ and $\phi_{h_2}$ commute for $h_1,h_2\in C^\infty(S^{n-1})$, so do the Lie derivatives with respect to the corresponding vector fields.

	\begin{corollary}\label{corollary:MixedSurfaceAreameasureDiffForm}
		If $L\in\K(\R^n)$ is a smooth convex body with strictly positive Gauss curvature, then for all $\omega\in\Omega^{n-1}(\R^n\times S^{n-1})^{tr}$ and $K\in\K(\R^n)$, 
		\begin{align*}
			\frac{d}{dt}\Big|_0\Phi_\omega(K+tL)=\Phi_{\mathcal{L}_{X_{h_L}}\omega}(K).
		\end{align*}
		In particular, there is a constant $c_k$ such that for $L_1,\dots,L_{n-k-1}\in\K(\R^n)$ smooth with strictly positive Gauss curvature,
		\begin{align*}
			\int_{S^{n-1}}\phi S(K[k],L_1,\dots,L_{n-k-1})=c_k\int_{\nc(K)}\pi_2^*\phi \wedge\mathcal{L}_{X_{h_{L_1}}}\dots\mathcal{L}_{X_{h_{L_{n-k-1}}}}i_T\pi_1^*\vol.
		\end{align*}
	\end{corollary}
	\begin{proof}
		If $L$ is a smooth convex body with strictly positive Gauss curvature, then 
		\begin{align*}
			\nc(K+L)=(\phi_{h_L})_*\nc(K)\quad\text{for}~K\in\K(\R^n),
		\end{align*}
		so in particular
		\begin{align*}
			\frac{d}{dt}\Big|_0\Phi_\omega(K+tL;\phi)=&\frac{d}{dt}\Big|_0\int_{\nc(K)}\pi_2^*\phi \wedge \phi_{th_L}^*\omega
			=\int_{\nc(K)}\pi_2^*\phi \wedge \mathcal{L}_{X_{h_L}}\omega\\
			=&\Phi_{\mathcal{L}_{X_{h_L}}\omega}(K;\phi),
		\end{align*}
		where we may interchange integration and differentiation since $\nc(K)$ is compact. The last equation follows inductively using that $S_{n-1}=\Phi_{i_T\pi_1^*\vol}$.
	\end{proof}
	
	For a positive definite symmetric endomorphism $A\in\Sym^2(\R^n)$, let $\mathcal{E}_A$ denote the ellipsoid with support function $h_{\mathcal{E}_A}(v)=\sqrt{\langle v,Av\rangle}$. Fix an orthonormal basis on $\R^n$ and let $E_{ij}\in \Sym^2(\R^n)$ denote the matrix whose coordinate representation has $1$ as its $(i,j)$th and $(j,i)$th entry and $0$ everywhere else. If $t>0$ is small enough, the matrix $A_{ij}(t)=Id_n+tE_{ij}$ is positive definite and consequently defines an ellipsoid in $\R^n$. For $h\in C^2(S^{n-1})$, we denote by $D^2h(v)$ the restriction of the Hessian of its $1$-homogeneous extension to $T_vS^{n-1}$.
	\begin{lemma}[\cite{KnoerrSmoothvaluationsconvex2025}*{Lemma 3.3}]
		\label{lemma:ExistenceEllipsoidsSpan}
		There exists $t\ge 1$ such that $D^2h_{\mathcal{E}_{A_{ij}(t)}}(v)$, $1\le i\le j\le n$, and $D^2h_{\mathcal{E}_{Id_n}}(v)$ span $\Sym^2 (T_vS^{n-1})$ for all $v\in S^{n-1}$.
	\end{lemma}
		
	\begin{proposition}\label{proposition:J_finitelyGeneratedEllipsoids}
		Let $\mathcal{E}_1,\dots, \mathcal{E}_N\in\K(\R^n)$ be ellipsoids such that $D^2h_{\mathcal{E}_j}(v)$, $1\le j\le N$, span $\Sym^2 (T_vS^{n-1})$ for all $v\in S^{n-1}$. Then for $0\le k\le n-1$,
		$\mathcal{J}^{n-k,k}(\R^n)^{tr}$ is generated as a $C^\infty(S^{n-1})$-module by the forms
		\begin{align}
			\label{eq:generatorsModule}
			\mathcal{L}_{X_{h_{\mathcal{E}_{j_1}}}}\dots\mathcal{L}_{X_{h_{\mathcal{E}_{j_k}}}}\pi_1^*\vol,\quad 1\le j_1,\dots,j_k\le N.
		\end{align}
	\end{proposition}
	\begin{proof}
		We use induction on $k$, where for $k=0$, we have $\mathcal{J}^{n,0}(\R^n)^{tr}=C^\infty(S^{n-1})\pi_1^*\vol$, so there is nothing to show. Assume that $k\ge 1$ and that the claim holds for $\mathcal{J}^{n-k+1,k-1}(\R^n)^{tr}$. Let us denote the $C^\infty(S^{n-1})$-module generated by the forms in Eq.~\eqref{eq:generatorsModule} by $A$. We claim that $A$ contains a $\GL(n,\R)$-invariant submodule. First note that the Lie derivatives $\mathcal{L}_{X_{h_{\mathcal{E}_j}}}$ commute with the multiplication with elements of $C(S^{n-1})$ since $\mathcal{L}_{X_{h_{\mathcal{E}_j}}}f=0$ for $f\in C^\infty(S^{n-1})$. It is therefore sufficient to show that
		\begin{align*}
			\mathcal{L}_{X_{h_{\mathcal{E}}}}(\mathcal{J}^{n-k+1,k-1}(\R^n)^{tr})\subset A
		\end{align*}
		for every ellipsoid $\mathcal{E}\in \K(\R^n)$, since 
		\begin{align*}
			\bigcup_{\mathcal{E}\in\K(\R^n)~\text{ellipsoid}}\mathcal{L}_{X_{h_{\mathcal{E}}}}(\mathcal{J}^{n-k+1,k-1}(\R^n)^{tr})
		\end{align*}
		spans a $\GL(n,\R)$-invariant $C^\infty(S^{n-1})$-module of $\mathcal{J}^{n-k,k}(\R^n)^{tr}$. \\
		Note that
		\begin{align*}
			\mathcal{L}_{X_{h_{\mathcal{E}}}}dx_l=di_{X_{h_{\mathcal{E}}}}dx_l=d\left(\frac{\partial h_{\mathcal{E}}(v)}{\partial v_l}\right)=\sum_{i=1}^n[D^2h_{\mathcal{E}}(v)]_{il}dv_i.
		\end{align*}
		Since $D^2h_{\mathcal{E}_j}$, $1\le j\le N$, span $\Sym^2(T_vS^{n-1})$ at each point $v\in S^{n-1}$, we find smooth functions $f_j\in C^\infty(S^{n-1})$ such that 
		\begin{align*}
			D^2h_{\mathcal{E}}=\sum_{j=1}^N f_j\cdot D^2h_{\mathcal{E}_j}.
		\end{align*}
		In particular, 
		\begin{align*}
			\mathcal{L}_{X_{h_{\mathcal{E}}}}dx_l=\sum_{i=1}^n\sum_{j=1}^N f_j\cdot [D^2h_{\mathcal{E}_j}]_{il}dv_i=\sum_{j=1}^Nf_j\mathcal{L}_{X_{h_{\mathcal{E}_j}}}dx_l.
		\end{align*}
		Since this holds for arbitrary $1\le l\le n$, the Leibniz rule implies
		\begin{align*}
			\mathcal{L}_{X_{h_{\mathcal{E}}}}\pi_1^*\vol =\sum_{j=1}^Nf_j\mathcal{L}_{X_{h_{\mathcal{E}_j}}}\pi_1^*\vol.
		\end{align*}
		As $k\ge 1$, the induction assumption implies that $\mathcal{J}^{n-k+1,k-1}(\R^n)^{tr}$ coincides with the $C^\infty(S^{n-1})$-module generated by $\mathcal{L}_{X_{h_{\mathcal{E}_{j_1}}}}\dots\mathcal{L}_{X_{h_{\mathcal{E}_{j_{k-1}}}}}\pi_1^*\vol$ for $1\le j_1,\dots,j_{k-1}\le n$, so it is sufficient to show that
		\begin{align*}
				\mathcal{L}_{X_{h_{\mathcal{E}}}}\left(\mathcal{L}_{X_{h_{\mathcal{E}_{j_1}}}}\dots\mathcal{L}_{X_{h_{\mathcal{E}_{j_{k-1}}}}}\pi_1^*\vol\right)\in A.
		\end{align*}
		However, since the Lie derivatives commute, we have
		\begin{align*}
			&\mathcal{L}_{X_{h_{\mathcal{E}}}}\left(\mathcal{L}_{X_{h_{\mathcal{E}_{j_1}}}}\dots\mathcal{L}_{X_{h_{\mathcal{E}_{j_{k-1}}}}}\pi_1^*\vol\right)=\mathcal{L}_{X_{h_{\mathcal{E}_{j_1}}}}\dots\mathcal{L}_{X_{h_{\mathcal{E}_{j_{k-1}}}}}\mathcal{L}_{X_{h_{\mathcal{E}}}}\pi_1^*\vol\\
			=&\sum_{j=1}^Nf_j\mathcal{L}_{X_{h_{\mathcal{E}_{j_1}}}}\dots\mathcal{L}_{X_{h_{\mathcal{E}_{j_{k-1}}}}}\mathcal{L}_{X_{h_{\mathcal{E}_j}}}\pi_1^*\vol\in A,
		\end{align*}
		where we used again that the Lie derivatives commute with the $C^\infty(S^{n-1})$-module structure. We thus see that $A$ contains a nontrivial $\GL(n,\R)$-invariant $C^\infty(S^{n-1})$-module of $\mathcal{J}^{n-k,k}(\R^n)^{tr}$, so \autoref{theorem:IrreducibilityDiffForms} shows that $A=\mathcal{J}^{n-k,k}(\R^n)^{tr}$.
	\end{proof}
	
	\begin{proof}[Proof of \autoref{maintheorem:finiteMixedArea}]
		If $\Psi\in \Area^\infty_k(\R^n)$, then there exists $\tau\in\mathcal{J}^{k+1,n-1-k}(\R^n)^{tr}$ such that $\Psi=\Phi_{i_T\tau}$ due to \autoref{corollary:PsiIsomorphism}. If we let $\mathcal{E}_1,\dots,\mathcal{E}_{N}$, $N=\binom{n+1}{2}+1$, denote ellipsoids as in \autoref{lemma:ExistenceEllipsoidsSpan}, then we may write
		\begin{align*}
			\tau=\sum_{1\le j_1,\dots,j_{n-k-1}\le N} g_{j_1,\dots,g_{j_{n-k-1}}}	\mathcal{L}_{X_{h_{\mathcal{E}_{j_1}}}}\dots\mathcal{L}_{X_{h_{\mathcal{E}_{j_{n-k-1}}}}}\pi_1^*\vol
		\end{align*} for suitable $g_{j_1},\dots,g_{j_{n-k-1}}\in C^\infty(S^{n-1})$ due to \autoref{proposition:J_finitelyGeneratedEllipsoids}. Since $\alpha\wedge\pi_1^*\vol=0$, we have $\pi_1^*\vol=\alpha\wedge i_T\pi_1^*\vol$ and as $\mathcal{L}_{X_{h_{\epsilon_j}}}\alpha=0$, this implies
		\begin{align*}
			i_T\tau\equiv&\sum_{1\le j_1,\dots,j_{n-k-1}\le N} g_{j_1,\dots,g_{j_{n-k-1}}}	\mathcal{L}_{X_{h_{\mathcal{E}_{j_1}}}}\dots\mathcal{L}_{X_{h_{\mathcal{E}_{j_{n-k-1}}}}}i_T\pi_1^*\vol\quad \mod \alpha.
		\end{align*}
		Since the normal cycle vanishes on multiples of $\alpha$, the claim follows directly from \autoref{corollary:MixedSurfaceAreameasureDiffForm}.
	\end{proof}
	As mentioned in the introduction, the previous description of the space $\mathcal{J}^n(\R^n)^{tr}$ may be used to obtain McMullen's Conjecture from the following description of $\GL(n,\R)$-smooth valuations in terms of differential forms. This description was originally obtained by Alesker in \cite{AleskerTheoryvaluationsmanifolds.2006}*{Theorem 5.2.1} as a consequence of the Irreducibilty Theorem \cite{AleskerDescriptiontranslationinvariant2001}. Recently, Hofstätter and the author found a different proof exploiting a corresponding regularity result for valuations on convex functions from \cite{KnoerrPaleyWienerSchwartz2025}. In particular, this proof does not require the Irreducibility Theorem, see \cite{HofstaetterKnoerrLocalizationvaluationsAleskers2025}*{Theorem A}.
	\begin{theorem}\label{theorem:SmoothValDiffForm}
		Let $0\le k\le n-1$. The following are equivalent for $\mu\in\Val_k(\R^n,\C)$:
		\begin{enumerate}
			\item $\mu$ is $\GL(n,\R)$-smooth, i.e. the map
				\begin{align*}
					\GL(n,\R)&\rightarrow\Val(\R^n,\C)\\
					g&\mapsto \left[K\mapsto \mu(g^{-1}K)\right]
				\end{align*}
				is smooth with respect to the natural Banach topology on $\Val(\R^n,\C)$.
			\item There exists a differential form $\omega\in \Omega^{k,n-1-k}(\R^n\times S^{n-1})^{tr}$ such that $\mu(K)=\int_{\nc(K)}\omega$.
		\end{enumerate}
	\end{theorem}
	\begin{remark}
		The result in \cite{HofstaetterKnoerrLocalizationvaluationsAleskers2025} is stated without the translation invariance condition on the differential forms (since the proof does not provide this property). That such a form can always be chosen to be translation invariant follows from results by Bernig and Bröcker \cite{BernigBroeckerValuationsmanifoldsRumin2007} (see also the discussion in \cite{HofstaetterKnoerrLocalizationvaluationsAleskers2025}*{Section~6.2}).
	\end{remark}
	The following was obtained in \cite{KnoerrSmoothvaluationsconvex2025}*{Theorem 3.4} as a consequence of Alesker's Irreducibility Theorem \cite{AleskerDescriptiontranslationinvariant2001}. We give a direct proof of this result using the characterization of $\GL(n,\R)$-smooth valuations in \autoref{theorem:SmoothValDiffForm}.	
	\begin{corollary}\label{corollary:finiteRepSmoothValMixedSurface}
			There exist $N:=\binom{n+1}{2}+1$ ellipsoids $\mathcal{E}_{1},\dots,\mathcal{E}_{N}$ such that for every $\GL(n,\R)$-smooth valuation $\mu\in\Val_k(\R^n,\C)$, $0\le k\le n-1$, there exist functions $g_{\alpha}\in C^\infty(S^{n-1})$, $\alpha\in \mathbb{N}^N$ with $|\alpha|=n-k-1$, such that
		\begin{align*}
			\mu(K)=\sum_{\substack{\alpha\in \mathbb{N}^N\\|\alpha|=n-k-1}}\int_{S^{n-1}}g_{\alpha}dS_{n-1}(K[k],\mathcal{E}_{1}[\alpha_1],\dots,\mathcal{E}_{N}[\alpha_N]).
		\end{align*}
	\end{corollary}
	\begin{proof}
		Since $\mu$ is a $\GL(n,\R)$-smooth valuation, \autoref{theorem:SmoothValDiffForm} shows that there exists a smooth differential form $\omega\in \Omega^{k,n-k-1}(\R^n\times S^{n-1})^{tr}$ such that $\mu(K)=\int_{\nc(K)}\omega$. Since $\nc(K)$ vanishes on multiples of $\alpha$ and $d\alpha$, we may again apply \autoref{corollary:RelationJtoKernel} and assume that $\omega=i_T\tau$ for some $\tau\in \mathcal{J}^{k+1,n-k-1}(\R^n)^{tr}$. Now the claim follows as in the proof of \autoref{maintheorem:finiteMixedArea} from \autoref{corollary:MixedSurfaceAreameasureDiffForm} and 
		 \autoref{proposition:J_finitelyGeneratedEllipsoids}.
	\end{proof}
	In particular, \autoref{theorem:SmoothValDiffForm} implies McMullen's Conjecture:
	\begin{corollary}\label{corollary:finitieCombMixedVolumes}
		Every $\GL(n,\R)$-smooth valuation in $\Val(\R^n,\C)$ is a finite linear combination of mixed volumes. In particular, mixed volumes span a dense subspace of $\Val(\R^n,\C)$.
	\end{corollary}
	\begin{proof}
		Since $\GL(n,\R)$-smooth valuations are dense in $\Val_k(\R^n,\C)$, it is sufficient to show that every $\GL(n,\R)$-smooth valuation is a finite linear combination of mixed volumes. If $\mu\in\Val_k(\R^n,\C)$ is a smooth valuation given as in \autoref{corollary:finiteRepSmoothValMixedSurface} by
		\begin{align*}
				\mu(K)=\sum_{\substack{\alpha\in \mathbb{N}^N\\|\alpha|=n-k-1}}\int_{S^{n-1}}g_{\alpha}dS_{n-1}(K[k],\mathcal{E}_{1}[\alpha_1],\dots,\mathcal{E}_{N}[\alpha_N])
		\end{align*}
		for suitable ellipsoids $\mathcal{E}_j$, $1\le j\le N$, and $g_\alpha\in C^\infty(S^{n-1})$, then we may write $g_\alpha=h_{K_\alpha}-h_{L_\alpha}$ for smooth convex bodies $K_\alpha, L_\alpha\in \K(\R^n)$  with strictly positive Gauss curvature, compare \cite{SchneiderConvexbodiesBrunn2014}*{Lemma~1.7.8}. Since
		\begin{align*}
			\int_{S^{n-1}}h_LdS_{n-1}(K[k],\mathcal{E}_{1}[\alpha_1],\dots,\mathcal{E}_{N}[\alpha_N])=nV(K[k],L,\mathcal{E}_{1}[\alpha_1],\dots,\mathcal{E}_{N}[\alpha_N]),
		\end{align*}
		where $V:(\K^n)^n\rightarrow\R$ denotes the mixed volume, the claim follows.
	\end{proof}
	
	\subsection{Regularity and locality properties of strongly $\GL(n,\R)$-continuous area measures}
	\label{section:LocalityCondition}
	As mentioned in the introduction, the area measures $S_k$, $0\le k\le n-1$, were originally characterized by Schneider. Since his conditions differ slightly from the ones we impose in \autoref{def:Area} and \autoref{maincorollary:InvariantAreaMeasures}, we include a short discussion of the differences. 
	\begin{theorem}[Schneider \cite{SchneiderKinematischeBeruhrmaefur1975}*{Theorem~2}]
		\label{theorem:SchneiderArea}
		Let $\Psi:\K(\R^n)\rightarrow\M(S^{n-1})$ be a map with the following properties:
		\begin{enumerate}
			\item $\Psi$ is translation invariant.
			\item $\Psi$ is $\SO(n)$-equivariant.
			\item $\Psi$ is continuous with respect to the weak* topology.
			\item $\Psi(\cdot; S^{n-1})$ is a valuation.
			\item If $B\subset S^{n-1}$ is a Borel set and $K,L\in\K(\R^n)$ are convex bodies such that $F(K,u)=F(L,u)$ for all $u\in B$, then $\Psi(K;B)=\Psi(L;B)$.
		\end{enumerate}
		Then there exist $c_0,\dots,c_{n-1}\in \C$ such that $\Psi=\sum_{k=0}^{n-1}c_k S_k$.
	\end{theorem}
	We already saw in \autoref{section:RelationValuationProperty} that (4) holds automatically for continuous locally determined functionals $\Psi:\K(\R^n)\rightarrow\M(S^{n-1})$, so the main difference between \autoref{theorem:SchneiderArea} and \autoref{def:Area}/\autoref{maincorollary:InvariantAreaMeasures} is the notion of locality in (5). It follows from the relation of the subdifferential of support functions to support sets (compare \autoref{section:preliminaries}) that locally determined functionals are precisely those maps $\Psi:\K(\R^n)\rightarrow\M(S^{n-1})$ satisfying condition (5) in \autoref{theorem:SchneiderArea} for all open subsets of $S^{n-1}$. In particular, any map $\Psi:\K(\R^n)\rightarrow \M(S^{n-1})$ satisfying (5) is automatically locally determined in the sense of \autoref{def:Area}. The main technical problem of condition (5) in \autoref{theorem:SchneiderArea} is that it is not stable under weak* convergence, which greatly limits any type of approximation argument. In more functional analytic terms, spaces of functionals of this type tend to lack good completeness properties, which underlie our approach. For locally determined functionals, these problems do not occur, which is why we use this slightly weaker notion of locality.\\
	
	Note that it is not clear whether arbitrary elements in $\Area(\R^n)$ satisfy condition (5) in \autoref{theorem:SchneiderArea} for general Borel sets, although we do not have a counterexample. However, we show below that strongly $\GL(n,\R)$-continuous area measures satisfy this stronger locality property. The proof relies on the stability of this property with respect to approximation in the strong topology on $\M(S^{n-1})$. We need the following lemma.
	
	\begin{lemma}\label{lemma:UniformConvergenceOperatorNorm}
		If $(\Psi_j)_j$ is a sequence in $\Area_k(\R^n)$ that converges to $\Psi\in\Area_k(\R^n)$ in the uniform strong topology, then for any bounded Borel function $f:S^{n-1}\rightarrow\C$, we have
		\begin{align*}
			\lim\limits_{j\rightarrow\infty}\Psi_j(K;f)=\Psi(K;f)
		\end{align*}
		for any $K\in\mathcal{K}(\R^n)$. Moreover, the convergence is uniform on bounded subsets of $\K(\R^n)$.
	\end{lemma}
	\begin{proof}
		Note that since any finite Borel measure on $S^{n-1}$ is a Radon measure, any $\nu\in\M(S^{n-1})$ satisfies
		\begin{align*}
			\sup_{f\in C(S^{n-1}):\|f\|_\infty\le 1}|\nu(f)|=
			\sup_{f:S^{n-1}\rightarrow\C~\text{measurable},~\|f\|_\infty\le 1}|\nu(f)|.
		\end{align*}
		Thus, for any bounded Borel function $f$,
		\begin{align*}
			|\Psi_j(K;f)-\Psi(K;f)|\le \|\Psi_j-\Psi\|\cdot\left(\frac{\diam(K)}{2}\right)^k \|f\|_\infty,
		\end{align*}
		which converges uniformly to $0$ on bounded subsets of $\K(\R^n)$.
	\end{proof}
	
	\begin{proposition}
		Let $\Psi\in\Area^0(\R^n)$, $B\subset S^{n-1}$ a Borel set, and $K,L\in\K(\R^n)$ two convex bodies such that $F(K,u)=F(L,u)$ for all $u\in B$. Then $\Psi(K;B)=\Psi(L;B)$.
	\end{proposition}
	\begin{proof}
		Due to the description of the normal cycle in Eq.~\eqref{eq:DefNormalCycle}, \autoref{lemma:SubdifferentialSupportFunction} and \autoref{corollary:RelationSubdiffNormalCycle}, the claim holds if $\Psi$ is $\GL(n,\R)$-smooth. If $\Psi\in \Area^0(\R^n)$, we may choose a sequence $(\Psi_j)_j$ in $\Area^\infty(\R^n)$ converging to $\Psi$ in the uniform strong topology by \autoref{proposition:DensitySmoothVectors}. \autoref{lemma:UniformConvergenceOperatorNorm} applied to the indicator function of $B$ shows that
		\begin{align*}
			\Psi(K;B)=\lim\limits_{j\rightarrow\infty}\Psi_j(K;B)=\lim\limits_{j\rightarrow\infty}\Psi_j(L;B)=\Psi(L;B).
		\end{align*}
	\end{proof}
	
	We use a similar argument to obtain the following regularity property.
	\begin{lemma}\label{lemma:AbsoluteContWRTLebesgue}
		Let $\Psi\in \Area^0(\R^n)$, $K\in\K(\R^n)$ a smooth convex body with strictly positive Gauss curvature. Then $\Psi(K)$ is absolutely continuous with respect to the spherical Lebesgue measure.
	\end{lemma}
	\begin{proof}
		If $\Psi$ is $\GL(n,\R)$-smooth, this follows directly from \autoref{maintheorem:IntegralRepresentation}, since $\Psi(K)$ is given by the pullback of a smooth differential form by the smooth map
		\begin{align*}
			G_K:S^{n-1}&\rightarrow \nc(K)\\
			v&\mapsto (\nu_K^{-1}(v),v),
		\end{align*}
		where $\nu_K:\partial K\rightarrow S^{n-1}$ denotes the Gauss map. In the general case, we may approximate $\Psi\in\Area^0(\R^n)$ in the uniform strong topology by a sequence of $\GL(n,\R)$-smooth area measures by \autoref{proposition:DensitySmoothVectors}, so $\Psi(K)$ is the limit of absolutely continuous measures in the strong topology and thus also absolutely continuous with respect to the spherical Lebesgue measures.
	\end{proof}
	Note that the previous result implies in particular that the example considered in \autoref{remark:NotStronglyCont} is not strongly $\GL(n,\R)$-continuous.

	\subsection{Injectivity of moment maps}\label{section:momentMaps}
	In this section, we consider an extension of certain injectivity results from \cite{BernigDualareameasures2019}. We will be interested in the following maps:\\
	
	The first variation is the map $\delta:\Val^\infty(\R^n,\C)\rightarrow\Area^\infty(\R^n)$ characterized by
	\begin{align*}
		\frac{d}{dt}\Big|_0\mu(K+tL)=\int_{S^{n-1}}h_Ld(\delta\mu(K)).
	\end{align*}
	We refer to \cite{Wannerermoduleunitarilyinvariant2014} for an interpretation in terms of differential forms. The first moment map is the map $M^1:\Area(\R^n)\rightarrow \Val(\R^n,\C)\otimes (\R^n)^*_\C$ characterized by
	\begin{align*}
		\left\langle [M^1(\Psi)](K),x\right\rangle= \int_{S^{n-1}}h_{\{x\}}(v)d\Psi(K;v)\quad\text{for}~x\in\R^n.
	\end{align*}
	We similarly define the higher order moment maps $M^r:\Area(\R^n)\rightarrow \Val(\R^n,\C)\otimes \Sym^r(\R^n)_\C$ by
	\begin{align*}
		[M^r(\Psi)](K)=\int_{S^{n-1}}v^{\otimes r}d\Psi(K;v).
	\end{align*}
	
	The following result was established in \cite{BernigDualareameasures2019}*{Theorem~1} for area measures given by integration with respect to the normal cycle. Due to \autoref{maintheorem:IntegralRepresentation}, it holds in fact for $\GL(n,\R)$-smooth area measures.
	\begin{theorem}[\cite{BernigDualareameasures2019}*{Theorem~1}]
		\label{theorem:MomentMapInjectiveSmoothCase}
		\begin{enumerate}
			\item The kernel of the first moment map $M^1: \Area^\infty_k(\R^n)\rightarrow\Val_k(\R^n,\C)\otimes (\R^n)^*_\C$ equals the image of the first variation map $\delta: \Val^\infty_{k+1}(\R^n,\C)\rightarrow\Area_k^\infty(\R^n)$ for each $0\le k\le n-1$.
			\item The higher moment maps 
			\begin{align*}
				M^r:\Area_k^\infty(\R^n)\rightarrow \Val_k(\R^n,\C)\otimes \Sym^r(\R^n)_\C,
			\end{align*} $r\ge 2$, $1\le k\le n-1$, are injective.
		\end{enumerate}
	\end{theorem}
	We obtain the following extension to general continuous area measures.
	\begin{theorem}\label{theorem:MomentMapContinuousCase}
		\begin{enumerate}
			\item The kernel of $M^1:\Area_k(\R^n)\rightarrow \Val_k(\R^n,\C)\otimes (\R^n)^*_\C$ is the closure of the image of $\delta:\Val^{\infty}_{k+1}(\R^n,\C)\rightarrow\Area_k(\R^n)$ with respect to the uniform weak* or uniform compact topology for each $0\le k\le n-1$.
			\item The higher moment maps 
			\begin{align*}
				M^r:\Area_k(\R^n)\rightarrow \Val_k(\R^n,\C) \otimes  \Sym^r(\R^n)_\C,
			\end{align*}$r\ge 2$, $1\le k\le n-1$, are injective.
		\end{enumerate}
	\end{theorem}
	\begin{proof} 
		Note that all moment maps are continuous with respect to the uniform weak* and uniform compact topology.\\
		The first moment map is in addition $\GL(n,\R)$-equivariant. In particular, its kernel is a $\GL(n,\R)$-invariant subspace that is closed with respect to the uniform weak* and uniform compact topology. By \autoref{proposition:DensitySmoothVectors}, $\GL(n,\R)$-smooth area measures are dense in $\ker M^1$ with respect to both topologies, so the claim follows from \autoref{theorem:MomentMapInjectiveSmoothCase}.\\
		Now let $l\ge 2$ and assume that $M^l$ is not injective. Since $M^l$ is continuous with respect to the uniform compact topology as well as $\SO(n)$-equivariant, its kernel $\ker M^l$ is a closed and $\SO(n)$-invariant subspace. In particular, $\ker M^l$ is a continuous representation of $\SO(n)$ with respect to the uniform compact topology,  which is quasi-complete (compare \cite{KnoerrPolynomiallocalfunctionals2025}*{Lemma 3.8}), so it contains a nontrivial $\SO(n)$-finite vector $\Psi\ne 0$. By \autoref{proposition:admissibility}, $\Psi$ is a $\GL(n,\R)$-smooth area measure, so \autoref{theorem:MomentMapInjectiveSmoothCase} implies $\Psi=0$, which is a contradiction. Thus $M^l$ must be injective.
	\end{proof}
	
\section{Remarks and open questions}\label{section:remarks}
	We close this article with some remarks on some future applications of these results to the integral geometry of area measures, as well as some open questions.\\
	
	A simple consequence of \autoref{maintheorem:IntegralRepresentation} is that the space $\Area^\infty(\R^n)$ is spanned by elements of the form $\phi\bullet S_{n-1}(\cdot+L)$ for $\phi\in C^\infty(S^{n-1})$ and $L\in\K(\R^n)$ smooth with strictly positive Gauss curvature. Similar to the convolution of smooth translation invariant valuations on convex bodies introduced by Bernig and Fu \cite{BernigFuConvolutionconvexvaluations2006}, this description gives rise to a product structure on $\Area^\infty(\R^n)$ by setting 
	\begin{align*}
		[\phi_1\bullet S_{n-1}(\cdot+L_1)]*[\phi_2	\bullet S_{n-1}(\cdot+L_2)]:=(\phi_1\phi_2)\bullet S_{n-1}(\cdot+L_1+L_2).
	\end{align*}
	It is not difficult to see that this is well defined and shares many properties of the convolution of smooth valuation. In fact, it is an extension of the $\Val^\infty(\R^n,\C)$-module structure on $\Area^\infty(\R^n)$ introduced by Wannerer \cite{Wannerermoduleunitarilyinvariant2014} and closely related to Bernig's convolution of so-called smooth dual area measures \cite{BernigDualareameasures2019}. We plan to discuss the precise relations and applications of this convolution product elsewhere.\\
	
	Note that we have a natural map $\mathrm{glob}:\Area_k(\R^n)\rightarrow\Val_k(\R^n,\C)$, $\Psi\mapsto \Psi(\cdot;S^{n-1})$, called the globalization map (compare \cite{WannererIntegralgeometryunitary2014,Wannerermoduleunitarilyinvariant2014}). This map has dense image, which contains the space of $\GL(n,\R)$-smooth valuations. 
	\begin{question}
		Is the map $\mathrm{glob}:\Area_k(\R^n)\rightarrow\Val(\R^n,\C)$ surjective? Is its restriction to the space of strongly $\GL(n,\R)$-continuous area measures surjective?
	\end{question}
	A similar questions concerns the corresponding space of continuous curvature measures, which were recently investigated by Schuhmacher and Wannerer \cite{SchuhmacherWannererTranslationinvariantcurvature2026} motivated by previous work by Weil \cite{WeilIntegralgeometrytranslation2015,WeilIntegralgeometrytranslation2017} (under the terminology of local functionals and their kernels).
	\begin{question}
		Does every element $\Psi\in \Area(\R^n)$ have the property that $\Psi(K;B)=\Psi(L,B)$ for all convex bodies $K,L\in\K(\R^n)$ and all Borel sets $B\subset S^{n-1}$ such that $F(K,u)=F(L,u)$ for $u\in B$?
	\end{question}
	The density criterion in \autoref{maintheorem:IrreducibleModule} uses both the action of $\GL(n,\R)$ and the module structure. From a purely representation theoretic perspective, this is not satisfactory, in particular, since \autoref{maintheorem:IntegralRepresentation} shows that $\Area(\R^n)$ has to be a rather simple representation of $\GL(n,\R)$. In particular, it is not difficult to see that $\Area_k(\R^n)$ has finite length (compare \autoref{theorem:MomentMapContinuousCase}). We can decompose $\Area_k(\R^n)=\Area_k^+(\R^n)\oplus\Area_k^-(\R^n)$ into even/odd area measures with respect of the action of $-Id_n\in\GL(n,\R)$, however, each of these spaces contains the kernel of the first moment map as a proper $\GL(n,\R)$-invariant subspace, so these spaces are not themselves irreducible.
	\begin{question}
		What is the length of $\Area_k(\R^n)$ as a $\GL(n,\R)$-module?
	\end{question}
	In \cite{SchuhmacherWannererTranslationinvariantcurvature2026}, Schuhmacher and Wannerer conjecture that the corresponding spaces of continuous translation invariant curvature measures should have length $2$. It seems to be reasonable to expect a similar answer in the case of continuous area measures.
\bibliographystyle{abbrv}
\bibliography{../../library/library.bib}

\begin{bibdiv}
\begin{biblist}

\bib{AbardiaEvequozBernigAdditivekinematicformulas2021}{article}{
      author={Abardia-Ev\'equoz, Judit},
      author={Bernig, Andreas},
       title={Additive kinematic formulas for flag area measures},
        date={2021},
        ISSN={0025-5831,1432-1807},
     journal={Math. Ann.},
      volume={381},
      number={3-4},
       pages={1615\ndash 1652},
         url={https://doi.org/10.1007/s00208-021-02264-w},
      review={\MR{4333426}},
}

\bib{AbardiaEvequozEtAlFlagareameasures2019}{article}{
      author={Abardia-Ev\'equoz, Judit},
      author={Bernig, Andreas},
      author={Dann, Susanna},
       title={Flag area measures},
        date={2019},
        ISSN={0025-5793,2041-7942},
     journal={Mathematika},
      volume={65},
      number={4},
       pages={958\ndash 989},
         url={https://doi.org/10.1112/s0025579319000226},
      review={\MR{3984044}},
}

\bib{AleskerP.McMullensconjecture2000}{article}{
      author={Alesker, Semyon},
       title={On {P}. {M}c{M}ullen's conjecture on translation invariant
  valuations},
        date={2000},
        ISSN={0001-8708},
     journal={Adv. Math.},
      volume={155},
      number={2},
       pages={239\ndash 263},
         url={https://doi.org/10.1006/aima.2000.1918},
      review={\MR{1794712}},
}

\bib{AleskerDescriptiontranslationinvariant2001}{article}{
      author={Alesker, Semyon},
       title={Description of translation invariant valuations on convex sets
  with solution of {P}. {M}c{M}ullen's conjecture},
        date={2001},
        ISSN={1016-443X},
     journal={Geom. Funct. Anal.},
      volume={11},
      number={2},
       pages={244\ndash 272},
         url={https://doi.org/10.1007/PL00001675},
      review={\MR{1837364}},
}

\bib{AleskerTheoryvaluationsmanifolds.2006}{article}{
      author={Alesker, Semyon},
       title={Theory of valuations on manifolds. {I}. {L}inear spaces},
        date={2006},
        ISSN={0021-2172},
     journal={Israel J. Math.},
      volume={156},
       pages={311\ndash 339},
         url={https://doi.org/10.1007/BF02773837},
      review={\MR{2282381}},
}

\bib{AleskerIntroductiontheoryvaluations2018}{book}{
      author={Alesker, Semyon},
       title={Introduction to the theory of valuations},
      series={CBMS Regional Conference Series in Mathematics},
   publisher={Conference Board of the Mathematical Sciences, Washington, DC; by
  the American Mathematical Society, Providence, RI},
        date={2018},
      volume={126},
        ISBN={978-1-4704-4359-7},
      review={\MR{3820854}},
}

\bib{AleskerFuTheoryvaluationsmanifolds.2008}{article}{
      author={Alesker, Semyon},
      author={Fu, Joseph H.~G.},
       title={Theory of valuations on manifolds. {III}. {M}ultiplicative
  structure in the general case},
        date={2008},
        ISSN={0002-9947},
     journal={Trans. Amer. Math. Soc.},
      volume={360},
      number={4},
       pages={1951\ndash 1981},
         url={https://doi.org/10.1090/S0002-9947-07-04489-3},
      review={\MR{2366970}},
}

\bib{BernigHadwigertypetheorem2009}{article}{
      author={Bernig, Andreas},
       title={A {H}adwiger-type theorem for the special unitary group},
        date={2009},
        ISSN={1016-443X},
     journal={Geom. Funct. Anal.},
      volume={19},
      number={2},
       pages={356\ndash 372},
         url={https://doi.org/10.1007/s00039-009-0008-4},
      review={\MR{2545241}},
}

\bib{BernigDualareameasures2019}{article}{
      author={Bernig, Andreas},
       title={Dual area measures and local additive kinematic formulas},
        date={2019},
        ISSN={0046-5755,1572-9168},
     journal={Geom. Dedicata},
      volume={203},
       pages={85\ndash 110},
         url={https://doi.org/10.1007/s10711-019-00427-3},
      review={\MR{4027586}},
}

\bib{BernigBroeckerValuationsmanifoldsRumin2007}{article}{
      author={Bernig, Andreas},
      author={Br\"{o}cker, Ludwig},
       title={Valuations on manifolds and {R}umin cohomology},
        date={2007},
        ISSN={0022-040X},
     journal={J. Differential Geom.},
      volume={75},
      number={3},
       pages={433\ndash 457},
      review={\MR{2301452}},
}

\bib{BernigFuConvolutionconvexvaluations2006}{article}{
      author={Bernig, Andreas},
      author={Fu, Joseph H.~G.},
       title={Convolution of convex valuations},
        date={2006},
        ISSN={0046-5755},
     journal={Geom. Dedicata},
      volume={123},
       pages={153\ndash 169},
         url={https://doi.org/10.1007/s10711-006-9115-7},
      review={\MR{2299731}},
}

\bib{BernigFuHermitianintegralgeometry2011}{article}{
      author={Bernig, Andreas},
      author={Fu, Joseph H.~G.},
       title={Hermitian integral geometry},
        date={2011},
        ISSN={0003-486X},
     journal={Ann. of Math. (2)},
      volume={173},
      number={2},
       pages={907\ndash 945},
         url={https://doi.org/10.4007/annals.2011.173.2.7},
      review={\MR{2776365}},
}

\bib{BernigSolanesKinematicformulasquaternionic2017}{article}{
      author={Bernig, Andreas},
      author={Solanes, Gil},
       title={Kinematic formulas on the quaternionic plane},
        date={2017},
        ISSN={0024-6115},
     journal={Proc. Lond. Math. Soc. (3)},
      volume={115},
      number={4},
       pages={725\ndash 762},
         url={https://doi.org/10.1112/plms.12050},
      review={\MR{3716941}},
}

\bib{BorelSomeremarksLie1949a}{article}{
      author={Borel, Armand},
       title={Some remarks about {L}ie groups transitive on spheres and tori},
        date={1949},
        ISSN={0002-9904},
     journal={Bull. Amer. Math. Soc.},
      volume={55},
       pages={580\ndash 587},
         url={https://doi.org/10.1090/S0002-9904-1949-09251-0},
      review={\MR{29915}},
}

\bib{CannasdaSilvaLecturessymplecticgeometry2001}{book}{
      author={Cannas~da Silva, Ana},
       title={Lectures on symplectic geometry},
      series={Lecture Notes in Mathematics},
   publisher={Springer-Verlag, Berlin},
        date={2001},
      volume={1764},
        ISBN={3-540-42195-5},
         url={https://doi.org/10.1007/978-3-540-45330-7},
      review={\MR{1853077}},
}

\bib{FuMongeAmperefunctions.1989}{article}{
      author={Fu, Joseph H.~G.},
       title={Monge-{A}mp\`ere functions. {I}},
        date={1989},
        ISSN={0022-2518},
     journal={Indiana Univ. Math. J.},
      volume={38},
      number={3},
       pages={711\ndash 743},
         url={https://doi.org/10.1512/iumj.1989.38.38035},
      review={\MR{1017333}},
}

\bib{FuAlgebraicintegralgeometry2014}{incollection}{
      author={Fu, Joseph H.~G.},
       title={Algebraic integral geometry},
        date={2014},
   booktitle={Integral geometry and valuations},
      series={Adv. Courses Math. CRM Barcelona},
   publisher={Birkh\"{a}user/Springer},
     address={Basel},
       pages={47\ndash 112},
      review={\MR{3330875}},
}

\bib{GaskproofSchwartzskernel1961}{article}{
      author={Gask, H.},
       title={A proof of {S}chwartz's kernel theorem},
        date={1960},
        ISSN={0025-5521},
     journal={Math. Scand.},
      volume={8},
       pages={327\ndash 332},
         url={https://doi.org/10.7146/math.scand.a-10614},
      review={\MR{125438}},
}

\bib{GoodeyWeilDistributionsvaluations1984}{article}{
      author={Goodey, Paul},
      author={Weil, Wolfgang},
       title={Distributions and valuations},
        date={1984},
        ISSN={0024-6115},
     journal={Proc. London Math. Soc. (3)},
      volume={49},
      number={3},
       pages={504\ndash 516},
         url={https://doi.org/10.1112/plms/s3-49.3.504},
      review={\MR{759301}},
}

\bib{HadwigerVorlesungenuberInhalt1957}{book}{
      author={Hadwiger, Hugo},
       title={Vorlesungen \"{u}ber {I}nhalt, {O}berfl\"{a}che und
  {I}soperimetrie},
   publisher={Springer},
     address={Berlin-G\"{o}ttingen-Heidelberg},
        date={1957},
      review={\MR{0102775}},
}

\bib{HofstaetterKnoerrLocalizationvaluationsAleskers2025}{article}{
      author={Hofst\"atter, Georg~C.},
      author={Knoerr, Jonas},
       title={Localization of valuations and {A}lesker's irreducibility
  theorem},
        date={2025},
     journal={arXiv:2510.25698},
         url={https://arxiv.org/abs/2510.25698},
}

\bib{HuybrechtsComplexgeometry2005}{book}{
      author={Huybrechts, Daniel},
       title={Complex geometry},
    subtitle={An introduction},
      series={Universitext},
   publisher={Springer-Verlag},
     address={Berlin},
        date={2005},
        ISBN={3-540-21290-6},
      review={\MR{2093043}},
}

\bib{Knoerrsupportduallyepi2021}{article}{
      author={Knoerr, Jonas},
       title={The support of dually epi-translation invariant valuations on
  convex functions},
        date={2021},
        ISSN={0022-1236},
     journal={J. Funct. Anal.},
      volume={281},
      number={5},
       pages={Paper No. 109059, 52},
         url={https://doi.org/10.1016/j.jfa.2021.109059},
      review={\MR{4252807}},
}

\bib{KnoerrMongeAmpereoperators2024}{article}{
      author={Knoerr, Jonas},
       title={Monge-{A}mp\`ere operators and valuations},
        date={2024},
        ISSN={0944-2669,1432-0835},
     journal={Calc. Var. Partial Differential Equations},
      volume={63},
      number={4},
       pages={Paper No. 89, 34},
         url={https://doi.org/10.1007/s00526-024-02698-5},
      review={\MR{4728223}},
}

\bib{KnoerrSmoothvaluationsconvex2024}{article}{
      author={Knoerr, Jonas},
       title={Smooth valuations on convex functions},
        date={2024},
        ISSN={0022-040X,1945-743X},
     journal={J. Differential Geom.},
      volume={126},
      number={2},
       pages={801\ndash 835},
         url={https://doi.org/10.4310/jdg/1712344223},
      review={\MR{4728707}},
}

\bib{KnoerrPaleyWienerSchwartz2025}{article}{
      author={Knoerr, Jonas},
       title={A {P}aley--{W}iener--{S}chwartz theorem for smooth valuations on
  convex functions},
        date={2025},
     journal={arXiv:2505.22464},
         url={https://arxiv.org/abs/2505.22464},
}

\bib{KnoerrPolynomiallocalfunctionals2025}{article}{
      author={Knoerr, Jonas},
       title={Polynomial local functionals on convex functions},
        date={2025},
     journal={arXiv:2512.15402},
         url={https://arxiv.org/abs/2512.15402},
}

\bib{KnoerrSmoothvaluationsconvex2025}{article}{
      author={Knoerr, Jonas},
       title={Smooth valuations on convex bodies and finite linear combinations
  of mixed volumes},
        date={2025},
     journal={Proc. London Math. Soc.},
      volume={130},
      number={6},
       pages={e70057},
  url={https://londmathsoc.onlinelibrary.wiley.com/doi/abs/10.1112/plms.70057},
}

\bib{KnoerrIntegralrepresentationpolynomial2026}{article}{
      author={Knoerr, Jonas},
       title={Integral representation of polynomial local functionals on convex
  functions},
        date={2026},
     journal={arXiv:2604.24485},
         url={https://arxiv.org/abs/2604.24485},
}

\bib{KotrbatyWannererIntegralgeometryoctonionic2025}{article}{
      author={Kotrbat\'y, Jan},
      author={Wannerer, Thomas},
       title={Integral geometry on the octonionic plane},
        date={2025},
        ISSN={0022-2518,1943-5258},
     journal={Indiana Univ. Math. J.},
      volume={74},
      number={2},
       pages={407\ndash 465},
         url={https://doi.org/10.1512/iumj.2025.74.60161},
      review={\MR{4915448}},
}

\bib{McMullenValuationsEulertype1977}{article}{
      author={McMullen, Peter},
       title={Valuations and {E}uler-type relations on certain classes of
  convex polytopes},
        date={1977},
        ISSN={0024-6115},
     journal={Proc. London Math. Soc. (3)},
      volume={35},
      number={1},
       pages={113\ndash 135},
         url={https://doi.org/10.1112/plms/s3-35.1.113},
      review={\MR{448239}},
}

\bib{McMullenContinuoustranslationinvariant1980}{article}{
      author={McMullen, Peter},
       title={Continuous translation-invariant valuations on the space of
  compact convex sets},
        date={1980},
        ISSN={0003-889X},
     journal={Arch. Math. (Basel)},
      volume={34},
      number={4},
       pages={377\ndash 384},
         url={https://doi.org/10.1007/BF01224974},
      review={\MR{593954}},
}

\bib{McMullenWeaklycontinuousvaluations1983}{article}{
      author={McMullen, Peter},
       title={Weakly continuous valuations on convex polytopes},
        date={1983},
        ISSN={0003-889X,1420-8938},
     journal={Arch. Math. (Basel)},
      volume={41},
      number={6},
       pages={555\ndash 564},
         url={https://doi.org/10.1007/BF01198585},
      review={\MR{731639}},
}

\bib{MontgomerySamelsonTransformationgroupsspheres1943}{article}{
      author={Montgomery, Deane},
      author={Samelson, Hans},
       title={Transformation groups of spheres},
        date={1943},
        ISSN={0003-486X},
     journal={Ann. of Math. (2)},
      volume={44},
       pages={454\ndash 470},
         url={https://doi.org/10.2307/1968975},
      review={\MR{8817}},
}

\bib{RockafellarConvexanalysis1997}{book}{
      author={Rockafellar, R.~Tyrrell},
       title={Convex analysis},
      series={Princeton Landmarks in Mathematics},
   publisher={Princeton University Press},
     address={Princeton, NJ},
        date={1997},
        ISBN={0-691-01586-4},
        note={Reprint of the 1970 original, Princeton Paperbacks},
      review={\MR{1451876}},
}

\bib{SchneiderKinematischeBeruhrmaefur1975}{article}{
      author={Schneider, Rolf},
       title={Kinematische {B}er\"uhrma\ss e{} f\"ur konvexe {K}\"orper},
        date={1975},
        ISSN={0025-5858,1865-8784},
     journal={Abh. Math. Sem. Univ. Hamburg},
      volume={44},
       pages={12\ndash 23},
         url={https://doi.org/10.1007/BF02992942},
      review={\MR{394447}},
}

\bib{SchneiderConvexbodiesBrunn2014}{book}{
      author={Schneider, Rolf},
       title={Convex bodies: the {B}runn-{M}inkowski theory},
     edition={expanded},
      series={Encyclopedia of Mathematics and its Applications},
   publisher={Cambridge University Press},
     address={Cambridge},
        date={2014},
      volume={151},
        ISBN={978-1-107-60101-7},
      review={\MR{3155183}},
}

\bib{SchuhmacherWannererTranslationinvariantcurvature2026}{article}{
      author={Schuhmacher, Jakob},
      author={Wannerer, Thomas},
       title={Translation invariant curvature measures of convex bodies},
        date={2026},
     journal={arXiv:2601.14193},
      eprint={2601.14193},
         url={https://arxiv.org/abs/2601.14193},
}

\bib{SepanskiCompactLiegroups2007}{book}{
      author={Sepanski, Mark~R},
       title={Compact {L}ie groups},
      series={Graduate Texts in Mathematics},
   publisher={Springer},
     address={New York},
        date={2007},
      volume={235},
        ISBN={978-0-387-30263-8; 0-387-30263-8},
         url={https://doi.org/10.1007/978-0-387-49158-5},
      review={\MR{2279709}},
}

\bib{WannererIntegralgeometryunitary2014}{article}{
      author={Wannerer, Thomas},
       title={Integral geometry of unitary area measures},
        date={2014},
        ISSN={0001-8708},
     journal={Adv. Math.},
      volume={263},
       pages={1\ndash 44},
         url={https://doi.org/10.1016/j.aim.2014.06.005},
      review={\MR{3239133}},
}

\bib{Wannerermoduleunitarilyinvariant2014}{article}{
      author={Wannerer, Thomas},
       title={The module of unitarily invariant area measures},
        date={2014},
        ISSN={0022-040X},
     journal={J. Differential Geom.},
      volume={96},
      number={1},
       pages={141\ndash 182},
      review={\MR{3161388}},
}

\bib{WarnerHarmonicanalysissemi1972}{book}{
      author={Warner, Garth},
       title={Harmonic analysis on semi-simple {L}ie groups. {I}},
      series={Die Grundlehren der mathematischen Wissenschaften},
   publisher={Springer-Verlag, New York-Heidelberg},
        date={1972},
      volume={Band 188},
      review={\MR{498999}},
}

\bib{WeilIntegralgeometrytranslation2015}{article}{
      author={Weil, Wolfgang},
       title={Integral geometry of translation invariant functionals, {I}:
  {T}he polytopal case},
        date={2015},
        ISSN={0196-8858,1090-2074},
     journal={Adv. in Appl. Math.},
      volume={66},
       pages={46\ndash 79},
         url={https://doi.org/10.1016/j.aam.2015.03.001},
      review={\MR{3328022}},
}

\bib{WeilIntegralgeometrytranslation2017}{article}{
      author={Weil, Wolfgang},
       title={Integral geometry of translation invariant functionals, {II}:
  {T}he case of general convex bodies},
        date={2017},
        ISSN={0196-8858,1090-2074},
     journal={Adv. in Appl. Math.},
      volume={83},
       pages={145\ndash 171},
         url={https://doi.org/10.1016/j.aam.2016.09.005},
      review={\MR{3573222}},
}

\end{biblist}
\end{bibdiv}

\Addresses
 \end{document}